\documentclass[a4paper,12pt]{article}
\pdfoutput=1
\usepackage{a4wide}
\usepackage{amssymb, amsmath}
\usepackage[amsmath,hyperref,thmmarks]{ntheorem}
\usepackage{t1enc,times}
\usepackage[utf8]{inputenc}
\usepackage[english]{babel}
\usepackage{enumerate}
\usepackage{hyperref}
\usepackage{natbib}
\usepackage{color,caption}
\emergencystretch=\hsize
\def\ca{{\mathcal A}}

\def\cf{{\mathcal F}}
\def\cg{{\mathcal G}}

\def\cl{{\mathcal L}}
\def\cn{{\mathcal N}}
\def\ct{{\mathcal T}}
\def\cs{{\mathcal S}}
\def\cx{{\mathcal X}}
\def\cw{{\mathcal W}}

\def\E{{\mathbb E}}

\def\N{{\mathbb N}}
\def\P{{\mathbb P}}

\def\R{{\mathbb R}}

\def\l{\ell}
\def\htau{\widehat{\tau}}
\def\t{\theta}
\def\htheta{\widehat{\theta}}
\def\hvt{\widehat{\vartheta}}
\def\nn{\cn\cn}
\def\s{\star}
\def\vt{\vartheta}

\def\ind#1{{\bf 1}_{\left\{#1\right\}}}

\def\abs#1{\left|#1\right|}

\def\inv#1{\mathop{\frac{1}{ #1}}\nolimits}
\def\expp#1{\mathop {\mathrm{e}^{ #1}}}

\theoremstyle{plain}
\newtheorem{theorem}{Theorem}[section]
\newtheorem{proposition}[theorem]{Proposition}
\newtheorem{lemma}[theorem]{Lemma}

\newtheorem{remarkh}[theorem]{Remark}
\newenvironment{remark}{\begin{remarkh}\rm}{\end{remarkh}}

\theoremstyle{nonumberplain}
\theoremheaderfont{\normalfont\itshape}
\theorembodyfont{\normalfont}
\theoremseparator{.}
\theoremsymbol{\ensuremath{\blacksquare}}
\newtheorem{proof}{Proof}

\makeatletter
\newcounter{hypo}
\renewcommand{\thehypo}{(${\mathcal H}$-\arabic{hypo})}
\newcommand{\dohypo}{\thehypo}
\newenvironment{hypo}[1][]{%
\refstepcounter{hypo}
\list{}{%
\settowidth{\labelwidth}{\dohypo}%
\setlength{\labelsep}{10pt}%
\setlength{\leftmargin}{\labelwidth}
\advance\leftmargin\labelsep%
}%
\item[\dohypo  #1]%
  }{%
\endlist
}

\makeatother

\title{Neural network regression\\ for Bermudan option pricing}
\date{\today}
\author{Bernard Lapeyre\thanks{Université Paris-Est, Cermics (ENPC), INRIA, F-77455 Marne-la-Vallée, France\newline
\indent email: \texttt{bernard.lapeyre@enpc.fr}} \and Jérôme Lelong
  \thanks{Univ. Grenoble Alpes, CNRS, Grenoble INP, LJK, 38000 Grenoble, France. \newline
\indent email: \texttt{jerome.lelong@univ-grenoble-alpes.fr}}}

\begin{document}
\maketitle

\begin{abstract}
  The pricing of Bermudan options amounts to solving a dynamic programming
  principle, in which the main difficulty, especially in high dimension, comes
  from the conditional expectation involved in the computation of the
  continuation value. These conditional expectations are classically computed by
  regression techniques on a finite dimensional vector space. In this work, we
  study neural networks approximations of conditional expectations.
  We prove the
  convergence of the well-known Longstaff and Schwartz algorithm when the
  standard least-square regression is replaced by a neural network
  approximation, assuming an efficient algorithm to compute this approximation.
  We illustrate the numerical efficiency of neural networks as an
  alternative to standard regression methods for approximating conditional
  expectations on several numerical examples.

  \noindent\textbf{Key words}: Bermudan options, optimal stopping,
  regression methods, deep learning, neural networks.

\end{abstract}

\section{Introduction}
\label{sec:intro}

Solving the backward recursion involved in the computation american
option prices has been a challenging problem for years and various
approaches have been proposed to approximate its solution. The real
difficulty lies in the computation of the conditional expectation
$\E[U_{T_{n+1}} | \cf_{T_n}]$ at each time step of the recursion. If
we were to classify the different approaches, we could say that there
are regression based approaches
(see~\cite{tilley93,carriere96,tsit:vanr:01, brgl04}) and quantization
approaches (see~\cite{bapa03,pages13_quantization}). We refer
to~\cite{BoucharWarin12} and~\cite{pages18numerical} for an in depth
survey of the different techniques to price Bermudan options.

Among all these available algorithms to compute american option prices using the
dynamic programming principle, the one proposed by~\cite{LS01} has the favour of
many practitioners. Their approach is based on iteratively selecting an optimal
policy. Here, we propose and analyse a version of this algorithm which uses
neural networks in order to compute an approximation of the conditional
expectation and then to obtain an optimal exercising policy.

The use of neural network for the computation of American option
prices is not new but we are aware of no work specifically devoted to
its use for LS-style algorithms (LS for~\cite{LS01}).
In \cite{haugh-kogan}, the authors used neural networks in numercial
experiments to price American options through the dynamic programing
equation on the value function.
This led them to a~\cite{tsit:vanr:01}-type algorithm which is
different from LS-type algorithm, studied in this paper,
which involve only the optimal stopping policy.
\cite{kohler10} used neural networks to price American options but
they also used the dynamic programing equation on the value
function. Moreover they used new samples of the whole path of the
underlying process $X$ at each time step $n$ to prove the convergence.
In our approach, we use a neural network inspired modification of the
original Longstaff-Schwartz algorithm and we draw a set of $M$ samples
with the distribution of $(X_{T_0},X_{T_1},\dots,X_{T_N} )$ before
starting and we use these very same samples at each time step. This
saves a lot of computational time by avoiding a very costly
resimulation at each time step, which very much improves the
efficiency of our approach. Deep learning was also used in the context
of optimal stopping by~\cite{becker2019deep, becker2019deep2} to
parametrize the optimal policy.

Now, we describe the framework of our study. We fix some finite time horizon
$T>0$ and a filtered probability space $(\Omega, \cf, (\cf_t)_{0 \le t \le T},
\P)$ modeling a financial market with $\cf_0$ being the trivial
$\sigma-$algebra. We assume that the short interest rate is modeled by an
adapted process $(r_t)_{0 \le t \le T}$ and that $\P$ is an associated risk
neutral measure. We consider a Bermudan option with exercising dates $0 = t_0
\le T_1 < T_2 < \dots < T_N = T$ and discounted payoff $\tilde Z_{T_n}$ if
exercised at time $T_n$. For convenience, we add $0$ and $T$ to the exercising
dates. This is definitely not a requirement of the method we propose here but it
makes notation lighter and avoids to deal with the purely European part involved
in the Bermudan option. We assume that the discrete time discounted payoff
process $(Z_{T_n})_{0 \le n \le N}$ is adapted to the filtration $(\cf_{T_n})_{0
\le n \le N}$ and that $\E[\max_{0 \le n \le N } \abs{Z_{T_n}}^2] < \infty$.

In a complete market, if $\E$ denote the expectation under the risk
neutral probability, standard arbitrage pricing arguments allows to
define the discounted value $(U_n)_{0 \le n \le N}$ of the Bermudan
option at times $(T_n)_{0 \le n \le N}$ by
\begin{equation}
  \label{eq:price}
  U_{T_n} = \sup_{\tau \in \ct_{T_n, T}} \E[Z_{\tau} | \cf_{T_n}].
\end{equation}
Using the Snell enveloppe theory, the sequence $U$ can be proved to be given by
the following dynamic programing equation
\begin{equation}
  \label{eq:dpp}
  \begin{cases}
    U_{T_N} & = Z_{T_N} \\
    U_{T_n} & = \max\left( Z_{T_n}, \E[U_{T_{n+1}} | \cf_{T_n}] \right), \quad 0 \le n  \le N-1.
  \end{cases}
\end{equation}
This equation can be rewritten in term of optimal policy.  Let
$\tau_n$ be the smallest optimal policy after time $T_n$ --- the
smallest stopping time reaching the supremum in~\eqref{eq:price} ---
then
\begin{equation}
  \label{eq:policy-iteration}
  \begin{cases}
    \tau_N = T_N \\
    \tau_n = T_n \ind{Z_{T_n} \ge \E[Z_{\tau_{n+1}} | \cf_{T_n}]}
            + \tau_{n+1} \ind{Z_{T_n} < \E[Z_{\tau_{n+1}} | \cf_{T_n}]}.
  \end{cases}
\end{equation}
All these methods based on the dynamic programming principle either as
value iteration~\eqref{eq:dpp} or policy
iteration~\eqref{eq:policy-iteration} require a Markovian setting to
be implemented such that the conditional expectation knowing the whole
past can be replaced by the conditional expectation knowing only the
value of a Markov process at the current time.
We assume that the discounted payoff process writes
$Z_{T_n} = \phi_n(X_{T_n})$, for any $0 \le n \le N$, where
$(X_t)_{0 \le t \le T}$ is an adapted Markov process taking values in
$\R^r$. Hence, the conditional expectation involved
in~\eqref{eq:policy-iteration} simplifies into
$\E[Z_{\tau_{n+1}} | \cf_{T_n}] = \E[Z_{\tau_{n+1}} | X_{T_n}]$ and
can therefore be approximated by a standard least square method.

Note that this setting allows to consider most standard financial
models.  For local volatility models, the process $X$ is typically
defined as $X_t = (r_t, S_t)$, where $S_t$ is the price of an asset
and $r_t$ the instantaneous interest rate (only $X_t = S_t$ when the
interest rate is deterministic). In the case of stochastic volatility
models, $X$ also includes the volatility process $\sigma$,
$X_t = (r_t, S_t, \sigma_t)$. Some path dependent options can also fit
in this framework at the expense of increasing the size of the process
$X$. For instance, in the case of an Asian option with payoff
$(\frac{1}{T} A_T - S_T)_+$ with $A_t = \int_0^t S_u du$, one can
define $X$ as $X_t = (r_t, S_t, \sigma_t, A_t)$ and then the Asian
option can be considered as a vanilla option on the two dimensional
but non tradable assets $(S, A)$.

Once the Markov process $X$ is identified, the conditional expectations can be written
\begin{equation}
  \label{eq:EMarkov}
 \E[Z_{\tau_{n+1}} | \cf_{T_n}]  = \E[Z_{\tau_{n+1}} | X_{T_n}] = \psi_n(X_{T_n})
\end{equation}
where $\psi_n$ solves the following minimization problem
\begin{equation*}
  \inf_{\psi \in L^2(\cl(X_{T_n}))} \E\left[\abs{Z_{\tau_{n+1}} - \psi(X_{T_n})}^2\right]
\end{equation*}
with $L^2(\cl(X_{T_n}))$ being the set of all measurable functions $f$ such that
$\E[f(X_{T_n})^2] < \infty$.  The real challenge comes from properly
approximating the space $L^2(\cl(X_{T_n}))$ by a finite dimensional space: one
typically uses polynomials or local bases (see~\cite{lemor_05,BoucharWarin12})
and in any case it always boils down to a linear regression. In this work, we
use neural networks to approximate $\psi_n$ in~\eqref{eq:EMarkov}. The main
difference between neural networks and the regression approaches commonly used
comes from the non linearity of neural networks, which also makes their
strength. Note that the set of neural networks with a fixed number of layers and
neurons is obviously not a vector space and not even convex. Through neural
networks, this paper investigates the effects of using non linear approximations
of conditional expectations in the Longstaff Schwartz algorithm.

The paper is organized as follows. In Section~\ref{sec:nn}, we start
with some preliminaries on neural networks and recall the universal
approximation theorem. Then, in Section~\ref{sec:algo}, we describe
our algorithm, whose convergence is studied in
Section~\ref{sec:cv}. Finally, we present some numerical results in
Section~\ref{sec:numerics}.

\section{Preliminaries on deep neural network}
\label{sec:nn}

Deep Neural networks (DNN) aim at approximating (complex non linear)
functions defined on finite-dimensional spaces, and in contrast with
the usual additive approximation theory built via basis functions,
like polynomials, they rely on composition of layers of simple
functions. The relevance of neural networks comes from the universal
approximation theorem and the Kolmogorov-Arnold representation theorem
(see~\cite{arnold57,kolmogorov56,cybenko89,hornik91,pinkus1999approximation}),
and this has shown to be successful in numerous practical
applications.

We consider the feed forward neural network --- also called multilayer
perceptron --- for the approximation of the continuation value at each
time step. From a mathematical point view, we can model a DNN by a non
linear function
\[
x \in \cx \subset \R^r \longmapsto \Phi(x; \theta) \in \R
\]
where $\Phi$ typically writes as function compositions. Let $L \ge 2$
be an integer, we write
\begin{equation}
  \label{eq:nn}
  \Phi = A_L \circ \sigma_{a} \circ A_{L-1} \circ \dots \circ \sigma_{a} \circ A_1
\end{equation}
where for $\l=1,\dots,L$, $A_\l : \R^{d_{\l-1}} \to \R^{d_\l}$ are affine functions
\begin{equation*}
  A_\l(x)  = \cw_\l x + \beta_\l \quad \mbox{ for } x \in \R^{d_{\l-1}},
\end{equation*}
with $\cw_\l \in \R^{d_{\l} \times d_{\l-1}}$, and $\beta_\l \in \R^{d_\l}$.  In
our setting, we have $d_1 = r$ and $d_L = 1$.  The function $\sigma_{a}$ is
often called the \emph{activation function} and is applied component wise. The
number $d_\l$ of rows of the matrix $\cw_\l$ is usually interpreted as the
number of neurons of layer $\l$.  For the sake of simpler notation, we embed all
the parameters of the different layers in a unique high dimensional parameter
$\theta$ and $\theta = (\cw_\l, \beta_\l)_{\l=1,\dots,L} \in \R^{N_d}$ with
$N_{d} = \sum_{\l = 1}^L d_\l (1 + d_{\l-1})$.

Let $L>0$ be fixed in the following, we introduce the set $\nn_\infty$
of all DNN of the above form. Now, we need to restrict the maximum
number of neurons per layer. Let $p \in \N$, $p>1$, we denote by
$\nn_p$ the set of neural networks with at most $p$ neurons per hidden layer
and $L-1$ layers and bounded parameters. More precisely, we pick an
increasing sequence of positive real numbers $(\gamma_p)_p$ such that
$\lim_{p \to \infty} \gamma_p = \infty$. We introduce the set
\begin{equation}
  \label{eq:dnn-parameters}
  \Theta_p = \{ \theta \in \R^r \times \R^{p \times r} \times {(\R^p \times \R^{p \times p})}^{L-2} \times \R \times \R^{p} \, : \, \abs{\theta} \le \gamma_p \}.
\end{equation}
Then, $\nn_p$ is defined by
\begin{align*}
  \cn\cn_p = \{ \Phi(\cdot; \theta ): \R^r \to \R \, ; \, \theta \in \Theta_p  \}
\end{align*}
and we have $\nn_\infty = \cup_{p \in \N} \nn_p$. An element of
$\nn_p$ with be denoted by $\Phi_p(\cdot; \theta)$ with
$\theta \in \Theta_p$. Note that the space $\nn_p$ is not a vector
space, nor a convex set and therefore finding the element of $\nn_p$
that best approximates a given function cannot be simply interpreted
as an orthogonal projection.

The use of DNN as function approximations is justified by the
fundamental results of~\cite{hornik91} (see
also~\cite{pinkus1999approximation} for related results).
\begin{theorem}[Universal Approximation Theorem]
  \label{thm:dnn-P}
  Assume that the function $\sigma_{a}$ is non constant and bounded. Let
  $\mu$ denote a probability measure on $\R^r$, then for any
  $L \ge 2$, $\nn_\infty$ is dense in $L^2(\R^r, \mu)$.
\end{theorem}
\begin{theorem}[Universal Approximation Theorem]
  \label{thm:dnn-cv}
  Assume that the function $\sigma_{a}$ is a non constant, bounded and
  continuous function, then, when $L=2$, $\nn_\infty$ is dense into
  $C(\R^r)$ for the topology of the uniform convergence on compact
  sets.
\end{theorem}

\begin{remark} \label{rem:dnn-cv} We can rephrase Theorem~\eqref{thm:dnn-P} in
  terms of approximating random variables. Let $Y$ be a real valued random
  variable defined on $(\Omega, \ca)$ s.t. $\E[Y^2] < \infty$. Let $X$ be an
  other random variable defined on $(\Omega, \ca)$ taking values in $\R^r$ and
  $\cg \subset \ca$ the smallest $\sigma-$algebra such that $X$ is $\cg$
  measurable. Then, there exists a sequence $(\theta_p)_{p \ge 2} \in
  \prod_{p=2}^\infty \Theta_p$, such that $\lim_{p \to \infty} \E[\abs{Y -
  \Phi_p(X; \theta_p)}^2] = 0$. Therefore, if for every $p \ge 2$, $\alpha_p \in
  \Theta_p$ solves
  \begin{align*}
    \inf_{\theta \in \Theta_p} \E[\abs{\Phi_p(X; \theta) - Y}^2],
  \end{align*}
  then the sequence $(\Phi_p(X; \alpha_p))_{p \ge 2}$ converges to
  $\E[Y|\cg]$ in $L^2(\Omega)$ when $p\to \infty$. Note that as long as the activation function $\sigma_a$ is bounded, $\Phi_p(X; \alpha_p) \in L^2(\Omega)$ for every $p \ge 2$.
\end{remark}

\section{The algorithm}
\label{sec:algo}

\subsection{Description of the algorithm}

We recall the dynamic programming principle on the optimal policy 
\begin{equation*}
  \begin{cases}
    \tau_N = T_N \\
    \tau_n = T_n \ind{Z_{T_n} \ge \E[Z_{\tau_{n+1}} | \cf_{T_n}]} + \tau_{n+1} \ind{Z_{T_n} < \E[Z_{\tau_{n+1}} | \cf_{T_n}]}, \; \text{for $1 \le n \le N-1$}.
  \end{cases}
\end{equation*}
Then, the time$-0$ price of the Bermudan option writes
\begin{equation*}
  U_0 = \max(Z_0, \E[Z_{\tau_1}]).
\end{equation*}
In order to solve this dynamic programming equation we need to
compute a conditional expectation at each time step. The
idea proposed by~\cite{LS01} was to approximate these conditional
expectations by a regression problem on a well chosen set of
functions. In this work, we use a DNN to perform this approximation.
\begin{equation}
  \label{eq:policy-iteration-dnn}
  \begin{cases}
    \tau_N^p = T_N^p \\
    \tau_n^p = T_n \ind{Z_{T_n} \ge \Phi_p(X_{T_n}; \theta_n^p)} +
    \tau_{n+1} \ind{Z_{T_n} < \Phi_p(X_{T_n}; \theta_n^p)}, \;
    \text{for $1 \le n \le N-1$}
  \end{cases}
\end{equation}
where $\theta_n^p$ solves the following optimization problem
\begin{align}
  \label{eq:theta_p}
  \inf_{\theta \in \Theta_p} \E\left[ \abs{\Phi_p(X_{T_n}; \theta) - Z_{\tau_{n+1}^p}}^2 \right].
\end{align}
Since the conditional expectation operator is an orthogonal
projection, we have
\begin{align*}
  \E\left[ \abs{\Phi_p(X_{T_n}; \theta) - Z_{\tau_{n+1}^p}}^2 \right] = & \E\left[ \abs{\Phi_p(X_{T_n}; \theta) - \E\left[Z_{\tau_{n+1}^p} | \cf_{T_n}\right] }^2 \right] \\
  & + \E\left[ \abs{Z_{\tau_{n+1}^p} - \E\left[Z_{\tau_{n+1}^p} | \cf_{T_n}\right] }^2\right].
\end{align*}
Therefore, any minimizer in~\eqref{eq:theta_p} is also a solution to
the following minimization problem
\begin{align}
  \label{eq:theta_p-cond}
  \inf_{\theta \in \Theta_p} \E\left[ \abs{\Phi_p(X_{T_n}; \theta) - \E\left[Z_{\tau_{n+1}^p} | \cf_{T_n} \right]}^2 \right].
\end{align}

The standard approach is to sample a bunch of paths of the model
$X^{(m)}_{T_0}, X^{(m)}_{T_1},\dots, X^{(m)}_{T_N}$ along with the
corresponding payoff paths
$Z^{(m)}_{T_0}, Z^{(m)}_{T_1},\dots, Z^{(m)}_{T_N}$, for
$m=1,\dots,M$. To compute the $\tau_n$'s on each path, one needs to
compute the conditional expectations $\E[Z_{\tau_{n+1}} | \cf_{T_n}]$
for $n=1,\dots,N-1$. Then, we introduce the final approximation of the
backward iteration policy, in which the truncated expansion is
computed using a Monte Carlo approximation
\begin{equation*}
  \begin{cases}
    \htau_N^{p, (m)} = T_N \\
    \htau_n^{p, (m)} = T_n \ind{Z_{T_n}^{(m)} \ge \Phi_{p}(X_{T_n}^{(m)}; \htheta_n^{p,M})} + \htau_{n+1}^{p,(m)} \ind{Z_{T_n}^{(m)} < \Phi_{p}^{(m)}(X_{T_n}^{(m)}; \htheta_n^{p,M})}, \; \text{for $1 \le n \le N-1$}
  \end{cases}
\end{equation*}
where $\htheta_n^{p,M}$ solves the sample average approximation of~\eqref{eq:theta_p}
\begin{align}
  \label{eq:theta_pM}
  \inf_{\theta \in \Theta_p} \inv{M} \sum_{m=1}^M \abs{\Phi_p(X_{T_n}^{(m)}; \theta) - Z^{(m)}_{\tau_{n+1}^{p, (m)}}}^2.
\end{align}

Then, we finally approximate the time$-0$ price of the option by
\begin{align}
  \label{eq:price-mc}
  U_0^{p,M} = \max\left(Z_0, \inv{M} \sum_{m=1}^M Z^{(m)}_{\htau_1^{p,(m)} }\right).
\end{align}

\begin{remark}
  Note that to implement the previous algorithm we need to compute a
  minimizer for the optimization
  problem~(\ref{eq:theta_pM}). Obviously this is not an easy task as
  this is a high-dimensional, non-convex and non smooth problem.

  It is usually solved in practice using toolboxes as
  \emph{Scikit-Learn} 
  or \emph{TensorFlow}, 
  by means of a stochastic gradient descent method for which a full
  convergence proof under realistic assumptions are still unknown in
  our knowledge. See~\cite{bottou2018optimization} or
  \cite{e2020mathematical} for recent in depth reviews of these
  subjects
  and~\cite{ghadimi2013stochastic}, 
  \cite{lei2019stochastic}, \cite{fehrman2020convergence}
  for 
  results for a non convex function.

\end{remark}

\section{Convergence of the algorithm}
\label{sec:cv}

We start this section on the study of the convergence by introducing
some bespoke notation following~\cite{clp02}.

\subsection{Notation}

First, it is important to note that the paths
$\tau_1^{p,(m)},\dots,\tau_N^{p,(m)}$ for $m=1,\dots,M$ are
identically distributed but not independent since the computations of
$\theta^p_n$ at each time step $n$ mix all the paths. We define the
vector $\vt$ of the coefficients of the successive expansions
$\vt^p = (\theta_1^p, \dots, \theta_{N-1}^p)$ and its Monte Carlo
counterpart
$\hvt^{p,M} = (\htheta_1^{p,M},.\dots,\htheta_{N-1}^{p,M})$.

Now, we recall the notation used by~\cite{clp02} to study the convergence of the original Longstaff Schwartz approach.\\
Given a deterministic parameter $t^p = (t_1^p, \dots, t_{N-1}^p)$ in ${\Theta_p}^{N-1}$ and deterministic vectors $z = z_1,\dots,z_N$ in $\R^N$ and $x = (x_1, \dots, x_N)$ in $(\R^r)^N$, we define the vector field $F= F_1, \dots, F_N$ by
\begin{align*}
  \begin{cases}
    F_N(t^p, z, x) &= z_N \\
    F_n(t^p, z, x) &= z_n \ind{z_n \ge \Phi_p(x_n; t^p_n)} +
    F_{n+1}(t^p, z, x) \ind{z_n < \Phi_p(x_n; t^p_n)}, \; \text{for
      $1 \le n \le N-1$}.
  \end{cases}
\end{align*}
Note that $F_n(t, z, x)$ does not depend on the first $n-1$ components
of $t^p$, ie $F_n(t^p, z, x)$ depends only $t^p_n, \dots,
t^p_{N-1}$. Moreover,
\begin{alignat*}{2}
  &F_n(\vt^p, Z, X) &&= Z_{\tau_n^p}, \\
  &F_n(\hvt^{p,M}, Z^{(m)}, X^{(m)}) &&= Z^{(m)}_{\htau_n^{p,(m)}}.
\end{alignat*}
Moreover, we clearly have that for all $t^p \in {\Theta_p}^{N-1}$
\begin{align}
  \label{eq:bound-F}
   \abs{F_n(t^p, Z, X)} & \le \max_{k \ge n} \abs{Z_{T_k}}.
\end{align}

\subsection{Deep neural network approximations of conditional expectations}

\begin{proposition}
  \label{prop:cv-prix-pn}
  Assume that $\E[\max_{0 \le n \le N } \abs{Z_{T_n}}^2] < \infty$. Then,
  $\lim_{p \to \infty} \E[Z_{\tau^{p}_n} | \cf_{T_n}] = \E[Z_{\tau_n}
  | \cf_{T_n}]$ in $L^2(\Omega)$ for all $1 \le n \le N$.
\end{proposition}
\begin{remark}
  Note that in the proof of Proposition~\ref{prop:cv-prix-pn}, there is no need
  for the sets $\Theta_p$ to be compact for every $p$. We could have chosen
  $\gamma_p = \infty$. However, the boundedness assumption will be required in
  the following section, so to work with the same approximations over the whole
  paper, we have decided to impose compactness on $\Theta_p$ for every $p$.
\end{remark}
\begin{proof}q
  We proceed by induction. The result is true for $n=N$ as
  $\tau_N = \tau^p_N = T$. Assume it holds for $n+1$
  (with $0 \le n \le N-1$), we will prove it is true for $n$. For this, using both recursion equations, we have
  \begin{align*}
    \E[Z_{\tau^p_n} - Z_{\tau_n} | \cf_{T_n}]
    & = Z_{T_n} \left(\ind{Z_{T_n} \ge \Phi_p(X_{T_n}; \theta^p_n)} - \ind{Z_{T_n} \ge \E[Z_{\tau_{n+1}}|\cf_{T_n}]}\right) \\
    & \qquad + \E\left[Z_{\tau_{n+1}^{p}} \ind{Z_{T_n} < \Phi_p(X_{T_n}; \theta^p_n)}   - Z_{\tau_{n+1}} \ind{Z_{T_n} < \E[Z_{\tau_{n+1}}|\cf_{T_n}]}| \cf_{T_n}\right].
  \end{align*}
  Now, defining $A_n^p$ as
  \[
    A_n^p = (Z_{T_n} - \E[Z_{\tau_{n+1}}|\cf_{T_n}])
    \left(\ind{Z_{T_n} \ge \Phi_p(X_{T_n}; \theta^p_n)} - \ind{Z_{T_n} \ge
        \E[Z_{\tau_{n+1}}|\cf_{T_n}]}\right),
  \]  
  we obtain
  \begin{equation}\label{eq:anp}
    \E[Z_{\tau^p_n} - Z_{\tau_n} | \cf_{T_n}] = A_n^p + \E\left[Z_{\tau_{n+1}^p} - Z_{\tau_{n+1}} | \cf_{T_n}\right] \ind{Z_{T_n} < \Phi_p(X_{T_n}; \theta^p_n)}.
  \end{equation}
  By the induction assumption, the term
  $\E\left[Z_{\tau_{n+1}^{p}} - Z_{\tau_{n+1}} | \cf_{T_n}\right]$
  goes to zero in $L^2(\Omega)$ as $p$ goes to infinity. So, we just
  have to prove that $A_n^p$ converges to zero in $L^2(\Omega)$ when $p \to \infty$. For this, note that
  \begin{align*}
    \abs{A_n^p} &\le \abs{Z_{T_n} - \E[Z_{\tau_{n+1}}|\cf_{T_n}]} \abs{\ind{Z_{T_n} \ge \Phi_p(X_{T_n}; \theta^p_n)} - \ind{Z_{T_n} \ge \E[Z_{\tau_{n+1}}|\cf_{T_n}]}}  \notag\\
    & \le \abs{Z_{T_n} - \E[Z_{\tau_{n+1}}|\cf_{T_n}]} \abs{ \ind{\E[Z_{\tau_{n+1}}|\cf_{T_n}] >Z_{T_n} \ge \Phi_p(X_{T_n}; \theta^p_n)} - \ind{\Phi_p(X_{T_n}; \theta^p_n) > Z_{T_n} \ge \E[Z_{\tau_{n+1}}|\cf_{T_n}]}} \notag\\
    & \le \abs{Z_{T_n} - \E[Z_{\tau_{n+1}}|\cf_{T_n}]} \ind{\abs{Z_{T_n} - \E[Z_{\tau_{n+1}}|\cf_{T_n}]} \le  \abs{\Phi_p(X_{T_n}; \theta^p_n) - \E[Z_{\tau_{n+1}}|\cf_{T_n}]}} \notag\\
                & \le \abs{\Phi_p(X_{T_n}; \theta^p_n) - \E[Z_{\tau_{n+1}}|\cf_{T_n}]}. \notag\\
  \end{align*}
  So we obtain
  \begin{equation}\label{eq:An}
    \abs{A_n^p} \le \abs{\Phi_p(X_{T_n}; \theta^p_n) - \E[Z_{\tau_{n+1}^p}|\cf_{T_n}]} +
      \abs{\E[Z_{\tau_{n+1}^p}|\cf_{T_n}] - \E[Z_{\tau_{n+1}}|\cf_{T_n}]}.
  \end{equation}
  Morevoer, as the conditional expectation is an orthogonal projection, we
  clearly have that
  \begin{align}
    \label{eq:orth-proj}
    \E\left[\abs{\E[Z_{\tau_{n+1}^{p}}|\cf_{T_n}] - \E[Z_{\tau_{n+1}}|\cf_{T_n}]}^2\right] &\le \E\left[\abs{\E[Z_{\tau_{n+1}^{p}}|\cf_{T_{n+1}}] - \E[Z_{\tau_{n+1}}|\cf_{T_{n+1}}]}^2\right].
  \end{align}
  Then, the induction assumption for $n+1$ yields that the second term on the r.h.s of~\eqref{eq:An} goes to zero in $L^2(\Omega)$ when $p \to \infty$.

  To deal with the first term on the r.h.s of~\eqref{eq:An}, we
  introduce for any $p \in \N$, $\tilde \theta_n^{p} \in \Theta_p$
  defined as a minimiser to
  \begin{align}
    \label{eq:theta_p-q--cond}
    \inf_{\theta \in \Theta_p} \E\left[ \abs{\Phi_p(X_{T_n}; \theta) - \E\left[Z_{\tau_{n+1}} | \cf_{T_n} \right]}^2 \right].
  \end{align}
  Note that $\Phi(\cdot, \tilde \theta_n^p)$ is the best approximation on
  $\nn_p$ of the true continuation value at time $n$. As $\theta^p_n$
  solves~\eqref{eq:theta_p-cond}, we clearly have that
  \begin{align}
    \label{eq:bound-pq}
    &\E\left[\abs{\Phi_p(X_{T_n}; \theta^p_n) - \E[Z_{\tau_{n+1}^p}|\cf_{T_n}]}^2\right] \nonumber\\
    & \le \E\left[\abs{\Phi_p(X_{T_n}; \tilde \theta^p_n) - \E[Z_{\tau_{n+1}^p}|\cf_{T_n}]}^2\right] \nonumber\\
    & \le 2 \E\left[\abs{\Phi_p(X_{T_n}; \tilde \theta^p_n) - \E[Z_{\tau_{n+1}}|\cf_{T_n}]}^2\right]  +
    2 \E\left[ \abs{\E[Z_{\tau_{n+1}}|\cf_{T_n}] - \E[Z_{\tau_{n+1}^p}|\cf_{T_n}]}^2\right]
  \end{align}
  Using the induction assumption for $n+1$, the second term on the
  r.h.s of~\eqref{eq:bound-pq} goes to zero in $L^2(\Omega)$ and
  from the universal approximation theorem (see
  Theorem~\ref{thm:dnn-cv} and Remark~\eqref{rem:dnn-cv}), we deduce
  that
  \begin{align*}
    \lim_{p \to \infty} \E\left[\abs{\Phi_p(X_{T_n}; \tilde \theta^p_n) - \E[Z_{\tau_{n+1}}|\cf_{T_n}]}^2\right] = 0.
  \end{align*}
  Then, we conclude that $\lim_{p \to \infty} \E[\abs{A_n^p}^2] = 0$.
\end{proof}
The next proposition show that if we have an estimate of the speed of
convergence for the network approximation for a suitable class of
functions we are able to derive the speed of convergence for the Bermudan option price.
\begin{proposition}
  Assume that for every $0 \le n \le N - 1$, there exists a sequence $(\delta_n^p)_p$ of positive real numbers such that
  \begin{align}
    \label{eq:NN-functional-rate}
    \inf_{\theta \in \Theta_p} \E\left[ \abs{\Phi_p(X_{T_n}; \theta) - \E\left[Z_{\tau_{n+1}} | \cf_{T_n} \right]}^2 \right] = O(\delta_n^p) \quad \mbox{when } p \to \infty.
  \end{align}
  Then, 
  \begin{equation*}
    \E\left[ \abs{\E[Z_{\tau^{p}_n} | \cf_{T_n}] - \E[Z_{\tau_n} | \cf_{T_n}]}^2 \right] = O\left(\sum_{i = n}^{N-1} \delta_i^p\right) \quad \mbox{when } p \to \infty.
  \end{equation*}
  
\end{proposition}
\begin{proof}
  We use the same notation as in the proof of Proposition~\ref{prop:cv-prix-pn}. We proceed by backward induction.

  First note that $\tau_N^p = \tau_N = T$ and then, using~\eqref{eq:anp},  $A_{N-1}^p=\E[Z_{\tau_{N-1}}^p - Z_{\tau_{N-1}} | \cf_{T_{N-1}}]$. Moreover, from~\eqref{eq:An}, we get that
  \begin{align*}
    \abs{A_{N-1}^p} \le \abs{\Phi_p(X_{T_{N-1}}; \theta^p_n) - \E[Z_{\tau_N^p}|\cf_{T_{N-1}}]}.
  \end{align*}
  Using again that $\tau_N^p = \tau_N$, we deduce that
  \begin{align*}
    \E\left[ \abs{\E[Z_{\tau^{p}_{N-1}} | \cf_{T_{N-1}}] - \E[Z_{\tau_{N-1}} | \cf_{T_{N-1}}]}^2 \right] \le \inf_{\theta \in \Theta_p} \E\left[ \abs{\Phi_p(X_{T_{N-1}}; \theta) - \E\left[Z_{\tau_N} | \cf_{T_{N-1}} \right]}^2 \right].
  \end{align*}
  Therefore, 
  \begin{align*}
    \E\left[ \abs{\E[Z_{\tau^{p}_{N-1}} | \cf_{T_{N-1}}] - \E[Z_{\tau_{N-1}} | \cf_{T_{N-1}}]}^2 \right] = O(\delta_{N-1}^p) \quad \text{when} \quad p \to \infty.
  \end{align*}
  Assume the result holds true for $n + 1$ (with $0 \le n \le N-2)$, we will prove it is true for $n$.
  \begin{align*}
    \E\left[\abs{\E[Z_{\tau^p_n} - Z_{\tau_n} | \cf_{T_n}] }^2\right] & \le 2 \E\left[\abs{\E\left[Z_{\tau_{n+1}^p} - Z_{\tau_{n+1}} | \cf_{T_n}\right]}^2\right] + 2 \E[\abs{A_n^p}^2] \\
   & \le 2 \E\left[\abs{\E\left[Z_{\tau_{n+1}^p} - Z_{\tau_{n+1}} | \cf_{T_{n+1}}\right]}^2\right] + 2 \E[\abs{A_n^p}^2]
  \end{align*}
  where we have used~\eqref{eq:orth-proj}. Then, using the induction assumption, we get
  \begin{align*}
    \E\left[\abs{\E[Z_{\tau^p_n} - Z_{\tau_n} | \cf_{T_n}] }^2\right] & \le O(\delta_{n+1}^p) + 2 \E[\abs{A_n^p}^2].
  \end{align*}
  From~\eqref{eq:An} and~\eqref{eq:orth-proj}, we have
  \begin{align*}
    \E[\abs{A_n^p}^2] & \le 2 \E\left[\abs{\Phi_p(X_{T_n}; \theta^p_n) - \E[Z_{\tau_{n+1}^p}|\cf_{T_n}]}^2\right] + 2 \E\left[\abs{\E[Z_{\tau_{n+1}^{p}}|\cf_{T_{n+1}}] - \E[Z_{\tau_{n+1}}|\cf_{T_{n+1}}]}^2\right]\\
    & \le 4 \inf_{\theta \in \Theta_p} \E\left[ \abs{\Phi_p(X_{T_n}; \theta) - \E\left[Z_{\tau_{n+1}} | \cf_{T_n} \right]}^2 \right] \\
    & \qquad + 6 \E\left[\abs{\E[Z_{\tau_{n+1}^{p}}|\cf_{T_{n+1}}] - \E[Z_{\tau_{n+1}}|\cf_{T_{n+1}}]}^2\right]
  \end{align*}
  where the last inequality comes from~\eqref{eq:bound-pq}.

  From the induction assumption, the second term is bounded by
  $O\left(\sum_{i=n+1}^{N-1} \delta_i^p\right)$. From~\eqref{eq:NN-functional-rate}, the first term is bounded by $O(\delta_n^p)$. Then, we conclude that when $p \to \infty$
  \[E\left[ \abs{\E[Z_{\tau^{p}_n} | \cf_{T_n}] - \E[Z_{\tau_n} |
      \cf_{T_n}]}^2 \right] = O\left(\sum_{i = n}^{N-1} \delta_i^p\right).\]
\end{proof}

\subsection{Convergence of the Monte Carlo approximation }

In the following, we assume that $p$ is fixed and we study the convergence with
respect to the number of samples $M$. First, we recall some important results on
the convergence of the solution of a sequence of optimization problems whose
cost functions converge.

\subsubsection{Convergence of optimization problems}

Consider a sequence of real valued functions $(f_n)_n$ defined on a
compact set $K \subset \R^d$. Define,
\begin{align*}
  v_n = \inf_{x \in K} f_n(x)
\end{align*}
and let $x_n$ be a sequence of minimizers
\begin{align*}
  f_n(x_n) = \inf_{x \in K} f_n(x).
\end{align*}
From~\cite[Chap.~2]{MR1241645}, we have the following result.
\begin{lemma}
  \label{lem:optim-cv}
  Assume that the sequence $(f_n)_n$ converges uniformly on $K$ to a
  continuous function $f$. Let $v^\s = \inf_{x \in K} f(x)$ and
  $\cs^\s = \{x \in K \,:\, f(x) = v^\s\}$. Then $v_n \to v^\s$ and
  $d(x_n, \cs^\s) \to 0$ a.s.
\end{lemma}

In the following, we will also make heavy use of the following result,
which is a restatement of the law of large numbers in Banach spaces,
see~\cite[Corollary 7.10, page 189]{ledouxtalagrand} or~\cite[Lemma
A1]{MR1241645}.
\begin{lemma}\label{lem:ulln}
  Let $(\xi_i)_{i\geq 1}$ be a sequence of i.i.d. $\R^m$-valued random
  vectors and $h:\R^d\times\R^m\rightarrow \R$ be a measurable
  function. Assume that
  \begin{itemize}
    \item a.s., $\theta\in\R^d\mapsto h(\theta,\xi_1)$ is continuous,
    \item  $\forall C>0,\;\E\left[\sup_{|\theta|\leq C}|h(\theta,\xi_1)|\right]<+\infty$.
  \end{itemize}
  Then, a.s.
  $\theta\in\R^d\mapsto\frac{1}{n}\sum_{i=1}^n h(\theta,\xi_i)$
  converges locally uniformly to the continuous function
  $\theta\in\R^d\mapsto\E[h(\theta,\xi_1)]$, ie
  \[ \lim_{n \to \infty} \sup_{|\theta|\leq C}
    \abs{\frac{1}{n}\sum_{i=1}^n h(\theta,\xi_i) - \E[h(\theta,
      \xi_1)]} = 0 \, a.s.
  \]
\end{lemma}

\subsubsection{Strong law of large numbers}
To prove a strong law of large numbers, we will need the following assumptions.
\begin{hypo}
  \label{hyp:NN-smooth}
  For every $p \in \N$, $p>1$, there exist $q \ge 1$ and $\kappa_p > 0$ s.t.
  \begin{equation*}
    \forall \, x\in\R^r, \; \forall \, \theta \in \Theta_p, \quad \abs{\Phi_p(x, \theta)} \le \kappa_p (1 + \abs{x}^q).
  \end{equation*}
  Moreover, for all $1 \le n \le N-1$, a.s. the random functions
  $\theta \in \Theta_p \longmapsto \Phi_p(X_{T_n}, \theta)$ are
  continuous. Note that as $\Theta_p$ is a compact set, the continuity
  automatically yields the uniform continuity.
\end{hypo}
\begin{hypo}
  \label{hyp:int}
  For $q$ defined in~\ref{hyp:NN-smooth}, $\E[\abs{X_{T_n}}^{2q}] < \infty$ for all $0 \le n \le N$.
\end{hypo}
\begin{hypo}
  \label{hyp:P-nul}
  For all $p \in \N$, $p>1$ and all $1 \le n \le N-1$,
  $\P\left( Z_{T_n} = \Phi_p(X_{T_n}; \t^{p}_n) \right) = 0$.
\end{hypo}
We introduce the notation
\begin{align}\label{eq:minp}
  \cs^{p}_n = \arg\inf_{\theta \in \Theta_p} \E\left[ \abs{\Phi_p(X_{T_n}; \theta) - Z_{\tau_{n+1}^p}}^2 \right].
\end{align}
Note that $\cs^{p}_n$ is a non void compact set.
\begin{hypo}
  \label{hyp:uniqueNN}
  For every $p \in \N$, $p>1$ and every $1 \le n \le N$, for all $\theta^1, \theta^2 \in \cs^p_n$,
  \begin{equation*}
    \Phi_p(x ; \theta^1)  = \Phi_p(x ; \theta^2) \quad \mbox{for all } x \in \R^r
  \end{equation*}
\end{hypo}
\begin{remark}
  Assumption~\ref{hyp:NN-smooth} is clearly satisfied for the
  classical activation functions ReLU $\sigma_{a}(x) = (x)_+$, sigmoid
  $\sigma_{a}(x) = (1 + \expp{-x})^{-1}$ and $\sigma_{a}(x) = \tanh(x)$. When
  the law of $X_{T_n}$ has a density with respect to the Lebesgue
  measure, the continuity assumption stated in~\ref{hyp:NN-smooth} is
  even satisfied by the binary step activation function
  $\sigma_{a}(x) = \ind{x \ge 0}$.
\end{remark}
\begin{remark}
  Considering the natural symmetries existing in a neural network, it
  is clear that the set $\cs^p_n$ will hardly ever be reduced to a
  singleton. So, none of the parameters $\htheta^{p,M}_n$ or
  $\theta^p_n$ is unique. Here, we only require the function described
  by the neural network approximation to be unique but not its
  representation, which is much weaker and more realistic in
  practice. We refer to~\cite{albertini93,albertini94} for
  characterization of symmetries of neural networks and to~\cite{WH95}
  for results on existence and uniqueness of an optimal neural network
  approximation (but not its parameters).
\end{remark}

To start, we prove the convergence of the neural network approximation
of the conditional expectation at each time step.
\begin{proposition}
  \label{prop:PhithetaM}
  Assume that Assumptions~\ref{hyp:NN-smooth}-\ref{hyp:uniqueNN}
  hold. Let $\theta^p_n$ be a minimiser of optimization problem~(\ref{eq:minp})
  and $\htheta^{p,M}_n$ be a minimiser of  the sample average optimization problem~(\ref{eq:theta_pM}),
  then, for every $n=1,\dots,N$,
  $\Phi_p(X_{T_n}^{(1)};\htheta^{p,M}_n)$ converges to
  $\Phi_p(X_{T_n}^{(1)}; \theta^p_n)$ a.s. as $M \to \infty$.
\end{proposition}

\begin{lemma}
  \label{lem:Flip}
  For every $n=1,\dots, N-1$,
  \[
    \abs{F_n(a, Z, X) - F_n(b, Z, X)} \le \left(\sum_{i=n}^N
      \abs{Z_{T_i}}\right) \left(\sum_{i=n}^{N-1} \ind{\abs{Z_{T_i} -
          \Phi_p(X_{T_i}; b_i)} \le \abs{\Phi_p(X_{T_i}; a_i)
          -\Phi_p(X_{T_i}; b_i)} }\right)
  \]
\end{lemma}

\begin{proof}[Proof of Proposition~\ref{prop:PhithetaM}]
  We proceed by induction.  \\
  \noindent\emph{Step 1.} For $n=N-1$, $\htheta_{N-1}^{p,M}$ solves
  \begin{align*}
    \inf_{\theta \in \Theta_p} \sum_{m=1}^M \abs{\Phi_p(X_{T_{N-1}}^{(m)}; \theta) - Z^{(m)}_{T_N}}^2.
  \end{align*}
  We aim at applying Lemma~\ref{lem:ulln} to the sequence of
  i.i.d. random functions
  $h_m(\theta) = \abs{\Phi_p(X_{T_{N-1}}^{(m)}; \theta) -
    Z^{(m)}_{T_N}}^2$. From Assumptions~\ref{hyp:NN-smooth}
  and~\ref{hyp:int}, we deduce that
  \begin{align*}
    \E\left[\sup_{\theta \in \Theta_P} \abs{h_m(\theta)}\right] \le 2 \kappa_p \E[\abs{X_{T_{N-1}}}^{2q}] + \E[(Z_T)^2] < \infty.
  \end{align*}
  Then, Lemma~\ref{lem:ulln} implies that a.s. the function
  \[
    \theta \in \Theta_p \longmapsto \inv{M} \sum_{m=1}^M
    \abs{\Phi_p(X_{T_{N-1}}^{(m)}; \theta) - Z^{(m)}_{T_N}}^2
  \]
  converges uniformly to
  $\E[\abs{\Phi_p(X_{T_{N-1}}; \theta) - Z_{T_N}}^2]$. Hence, we
  deduce from Lemma~\ref{lem:optim-cv} that
  $d(\htheta_{N-1}^{p,M}, S^p_{N-1}) \to 0$ a.s. when $M \to
  \infty$. We restrict to a subset with probability one of the original probability space on which this convergence holds and the random functions
  $\Phi(X_{T_{N-1}}^{(1)}; \cdot)$ are uniformly
  continuous, see~\ref{hyp:NN-smooth}. There exists a sequence $(\xi^{p,M}_{N-1})_M$ taking values in $S^p_{N-1}$ such that
  $\abs{\htheta_{N-1}^{p,M} - \xi^{p,M}_{N-1}} \to 0$, when
  $M \to \infty$. The uniform continuity of the random functions
  $\Phi(X_{T_{N-1}}^{(1)}; \cdot)$ yields that
  \begin{align*}
    \Phi(X_{T_{N-1}}^{(1)};\htheta^{p,M}_{N-1}) - \Phi(X_{T_{N-1}}^{(1)};\xi^{p,M}_{N-1}) \to 0
  \end{align*}
  Then, we conclude from Assumption~\ref{hyp:uniqueNN}, that $\Phi(X_{T_{N-1}}^{(1)};\htheta^{p,M}_{N-1}) \to \Phi(X_{T_{N-1}}^{(1)};\theta^{p}_{N-1})$ \\

  \noindent\emph{Step 2.} Choose $n \le N-2$ and assume that the convergence result holds for
  $n+1,\dots,N-1$, we aim at proving this is true for $n$. We recall
  that $\htheta_n^{p,M}$ solves
  \begin{align*}
    \inf_{\theta \in \Theta_p} \sum_{m=1}^M \abs{\Phi_p(X_{T_{n}}^{(m)}; \theta) - F_{n+1}(\hvt^{p,M}, Z^{(m)}, X^{(m)})}^2.
  \end{align*}
  We introduce the two random functions for $\theta \in \Theta_p$
  \begin{align*}
    \hat v^M(\theta) &= \inv{M} \sum_{m=1}^M \abs{\Phi_p(X_{T_{n}}^{(m)}; \theta) - F_{n+1}(\hvt^{p,M}, Z^{(m)}, X^{(m)})}^2 \\
    v^M(\theta) &= \inv{M} \sum_{m=1}^M \abs{\Phi_p(X_{T_{n}}^{(m)}; \theta) - F_{n+1}(\vt^{p}, Z^{(m)}, X^{(m)})}^2.
  \end{align*}
  The function $v^M$ clearly writes as the sum of i.i.d. random
  variables. Moreover, by combining~\eqref{eq:bound-F} and
  Assumptions~\ref{hyp:NN-smooth} and~\ref{hyp:int}, we obtain
  \begin{align*}
    \E\left[ \sup_{\theta \in \Theta_p}\abs{\Phi_p(X_{T_{n}}; \theta) - F_{n+1}(\vt^{p}, Z, X)}^2\right] \le 2 \kappa \E[1 + \abs{X_{T_n}}^{2q}] + \E\left[\max_{\ell \ge n+1} (Z_{T_\ell})^2\right] < \infty.
  \end{align*}
  Then, the sequence of random functions $v^M$ a.s. converges
  uniformly to the continuous function $v$ defined for
  $\theta \in \Theta_p$ by
  \[ v(\theta) = \E\left[ \abs{\Phi_p(X_{T_{n}}; \theta) -
        F_{n+1}(\vt^{p}, Z, X)}^2\right].
  \]
  It remains to prove that
  $\sup_{\theta \in \Theta_p} \abs{\hat v^M(\theta) - v^M(\theta)} \to
  0$ a.s. when $M \to \infty$.
  \begin{align*}
    &\abs{\hat v^M(\theta) - v^M(\theta)} \\
    & \le \inv{M} \sum_{m=1}^M \abs{2 \Phi_p(X_{T_{n}}^{(m)}; \theta) - F_{n+1}(\hvt^{p,M}, Z^{(m)}, X^{(m)}) - F_{n+1}(\vt^{p}, Z^{(m)}, X^{(m)})} \\
    & \qquad \qquad \abs{F_{n+1}(\hvt^{p,M}, Z^{(m)}, X^{(m)}) - F_{n+1}(\vt^{p}, Z^{(m)}, X^{(m)})} \\
    & \le \inv{M} \sum_{m=1}^M 2 \left(\kappa (1 + |X_{T_{n}}^{(m)}|^{q}) + \max_{\ell \ge n+1} |Z_{T_\ell}| \right)\\
    & \qquad \qquad \abs{F_{n+1}(\hvt^{p,M}, Z^{(m)}, X^{(m)}) - F_{n+1}(\vt^{p}, Z^{(m)}, X^{(m)})}
  \end{align*}
  where we have used~\eqref{eq:bound-F} and
  Assumptions~\ref{hyp:NN-smooth} and~\ref{hyp:int}. Then from
  Lemma~\ref{lem:Flip}, we can write
  \begin{align*}
    &\abs{\hat v^M(\theta) - v^M(\theta)} \\
    & \le \inv{M} \sum_{m=1}^M 2 \left(\kappa (1 + |X_{T_{n}}^{(m)}|^{q}) + \max_{\ell \ge k+1} |Z_{T_\ell}| \right)\\
    & \qquad \qquad \left(\sum_{i=n+1}^N \abs{Z_{T_i}^{(m)}}\right) \left(\sum_{i=n+1}^{N-1} \ind{\abs{Z_{T_i}^{(m)} - \Phi_p(X_{T_i}^{(m)}; \t^{p}_i)} \le \abs{\Phi_p(X_{T_i}^{(m)}; \htheta^{p,M}_i) -\Phi_p(X_{T_i}^{(m)}; \t^{p}_i)} }\right) \\
    & \le \inv{M} \sum_{m=1}^M C \left(\kappa_p (1 + |X_{T_{n}}^{(m)}|^{2q}) + \sum_{i=n+1}^N \abs{Z_{T_i}^{(m)}}^{2} \right)\\
    & \qquad \qquad \left(\sum_{i=n+1}^{N-1} \ind{\abs{Z_{T_i}^{(m)} - \Phi_p(X_{T_i}^{(m)}; \t^{p}_i)} \le \abs{\Phi_p(X_{T_i}^{(m)}; \htheta^{p,M}_i) -\Phi_p(X_{T_i}^{(m)}; \t^{p}_i)} }\right)
  \end{align*}
  where $C$ is a generic constant only depending on $\kappa_p$, $n$
  and $N$.

  Let $\varepsilon >0$. Using the induction assumption and the strong
  law of large numbers, we have
  \begin{align*}
    &\limsup_M \sup_{\theta \in \Theta_p} \abs{\hat v^M(\theta) - v^M(\theta)} \\
    & \le \limsup_M \inv{M} \sum_{m=1}^M C \left((1 + |X_{T_{n}}^{(m)}|^{2q}) + \sum_{i=n+1}^N \abs{Z_{T_i}^{(m)}}^{2} \right) \\
    & \qquad \qquad \qquad
      \left(\sum_{i=n+1}^{N-1} \ind{\abs{Z_{T_i}^{(m)} - \Phi_p(X_{T_i}^{(m)}; \t^{p}_i)} \le \varepsilon } \right) \\
    & \le C \E\left[\left((1 + |X_{T_{n}}|^{2q}) + \sum_{i=n+1}^N \abs{Z_{T_i}}^{2} \right)
      \left(\sum_{i=n+1}^{N-1} \ind{\abs{Z_{T_i} - \Phi_p(X_{T_i}; \t^{p}_i)} \le \varepsilon } \right) \right]
  \end{align*}
  From~\ref{hyp:P-nul}, we deduce that
  $\lim_{\varepsilon \to 0} \ind{\abs{Z_{T_i} - \Phi_p(X_{T_i};
      \t^{p}_i)} \le \varepsilon } = 0$ a.s. and we conclude that
  a.s. $\hat v^M - v^M$ converges to zero uniformly. As we have
  already proved that a.s. $v^M$ converges uniformly to the continuous
  function $v$, we deduce that a.s. $\hat v^M$ converges uniformly to
  $v$. From Lemma~\ref{lem:optim-cv}, we conclude that
  $d(\htheta^{p,M}_n, S^p_n) \to 0$ a.s. when $M \to \infty$. We restrict to a subset with probability one of the original probability space on which this convergence holds and the random functions $\Phi(X_{T_{N-1}}^{(1)}; \cdot)$ are uniformly continuous, see~\ref{hyp:NN-smooth}. There exists a sequence $(\xi^{p,M}_{n})_M$ taking values in $S^p_n$ such that $\abs{\htheta_{n}^{p,M} - \xi^{p,M}_{n}} \to 0$ when $M \to \infty$. The uniform continuity of the random functions $\Phi(X_{T_{N-1}}^{(1)}; \cdot)$ yields that
  \begin{align*}
    \Phi(X_{T_{N-1}}^{(1)};\htheta^{p,M}_{n}) - \Phi(X_{T_{n}}^{(1)};\xi^{p,M}_n) \to 0 \quad \mbox{when } M \to \infty.
  \end{align*}
  Then, we conclude from Assumption~\ref{hyp:uniqueNN}, that
  $\Phi(X_{T_{n}}^{(1)};\htheta^{p,M}_{n}) \to
  \Phi(X_{T_{n}}^{(1)};\theta^{p}_{n})$ when ${M \to \infty}$.
\end{proof}

Now that the convergence of the expansion is established, we can study
the convergence of $U^{p,M}_0$ to $U^{p}_0$ when $M \to \infty$.
\begin{theorem}
  \label{thm:slln}
  Assume that Assumptions~\ref{hyp:NN-smooth}-\ref{hyp:uniqueNN}
  hold. Then, for $\alpha=1,2$ and every $n=1,\dots,N$,
  \[
  \lim_{M \to \infty} \inv{M} \sum_{m=1}^M \left(Z_{\htau_n^{p,(m)}}^{(m)}\right)^\alpha = \E\left[\left(Z_{\tau_n^p}\right)^\alpha\right] \quad a.s.
  \]
\end{theorem}
\begin{proof}
  Note that $\E[(Z_{\tau_n^p})^\alpha] = \E[F_n(\vt^p, Z, X)^\alpha]$
  and by the strong law of large numbers
  \[
    \lim_{M \to \infty} \inv{M} \sum_{m=1}^M F_n(\vt^p, Z^{(m)}, X^{(m)})^\alpha = \E[F_n(\vt^p, Z, X)^\alpha] \quad a.s.
  \]
  Hence, we have to prove that
  \[
    \Delta F_M = \inv{M} \sum_{m=1}^M \left(F_n(\hvt^{p,M}, Z^{(m)},
      X^{(m)})^\alpha - F_n(\vt^p, Z^{(m)}, X^{(m)})^\alpha\right)
    \xrightarrow[M \to \infty]{a.s} 0.
  \]
  For any $x, y \in \R$, and $\alpha=1,2$,
  $\abs{x^\alpha - y^\alpha}= \abs{x - y} \abs{x^{q-1} + y^{q-1}}$.
  Using Lemma~\ref{lem:Flip} and that
  $\abs{F_n(\gamma, z, g)} \le \max_{n \le j \le N} \abs{z_j}$, we
  have
  \begin{align*}
    & \abs{\Delta F_M}  \le \inv{M} \sum_{m=1}^M  \abs{F_n(\hvt^{p,M}, Z^{(m)}, X^{(m)})^\alpha - F_n(\vt^p, Z^{(m)}, G^{(m)})^\alpha} \\
    & \le 2  \inv{M} \sum_{m=1}^M  \sum_{i=n}^N \max_{n \le j \le N} \abs{Z_{T_j}^{(m)}}\abs{Z_{T_{i+1}}^{(m)}} \left(\sum_{i=n}^{N-1} \ind{\abs{Z_{T_i}^{(m)} - \Phi_p(X_{T_i}^{(m)}; \t^{p}_i)} \le \abs{\Phi_p(X_{T_i}^{(m)}; \htheta^{p,M}_i) -\Phi_p(X_{T_i}^{(m)}; \t^{p}_i)} }\right)
   \end{align*}
   Using Proposition~\ref{prop:PhithetaM}, for all $i=n,\dots,N-1$,
   $\abs{\Phi_p(X_{T_i}^{(1)}; \htheta^{p,M}_i) -
     \Phi_p(X_{T_i}^{(1)}; \t^{p}_i)} \to 0$ when $M \to \infty$. Then
   for any $\varepsilon > 0$,
   \begin{align*}
     & \limsup_M \abs{\Delta F_M} \\
     & \le 2 \limsup_M \inv{M} \sum_{m=1}^M  \sum_{i=n}^N \max_{n \le j \le N} \abs{Z_{T_j}^{(m)}} \abs{Z_{T_{i+1}}^{(m)}} \left(\sum_{i=k}^{N-1} \ind{\abs{Z_{T_i}^{(m)} - \Phi_p(X_{T_i}^{(m)}; \t^{p}_i)} \le  \varepsilon } \right) \\
     & \le 2 \E\left[ \sum_{i=n}^N \max_{n \le j \le N} \abs{Z_{T_j}} \abs{Z_{T_{i+1}}} \left(\sum_{i=n}^{N-1} \ind{\abs{Z_{T_i} - \Phi_p(X_{T_i}^{(m)}; \t^{p}_i)} \le \varepsilon} \right) \right]
   \end{align*}
   where the last inequality follows from the strong law of larger
   numbers as $\E[\max_{n \le j \le N} \abs{Z_{T_j}}^2 ] < \infty$. We
   conclude that $\limsup_M \abs{\Delta F_M} = 0$ by letting
   $\varepsilon$ go to $0$ and by using~\ref{hyp:P-nul}.
\end{proof}
The case $\alpha=1$ proves the strong law of large numbers for the algorithm.
Note that solving the minimisation problem~\eqref{eq:theta_pM} mixes all stopped
paths $Z_{\htau_n^{p,(m)}}^{(m)}$, it is unlikely that the estimators $\inv{M}
\sum_{m=1}^M Z_{\htau_n^{p,(m)}}^{(m)}$ for $1 \le n \le N$ are unbiased. We
recall that $U_n^{p,M} = \inv{M} \sum_{m=1}^M F_n(\hvt^{p,M}, Z^{(m)},G^{(m)})$
and $Z_{\tau_n^p} = F_n(\vt^p, Z, X)$. Then,
\begin{align*}
  & \E\left[U_n^{p,M}\right]  - \E\left[Z_{\tau_n^p} \right] = \E\left[ \inv{M} \sum_{m=1}^M \left(F_n(\hvt^{p,M}, Z^{(m)}, X^{(m)}) - F_n(\vt^p, Z^{(m)}, X^{(m)})\right)\right] \\
& = \E\left[ F_n(\hvt^{p,M}, Z^{(1)}, X^{(1)}) - F_n(\vt^p, Z^{(1)}, X^{(1)}) \right]
\end{align*}
where we have used that all the random variables have the same
distribution.

\newpage

\section{Numerical experiments}
\label{sec:numerics}

In this section, we compare the results given by the standard
Longstaff Schwartz approach with polynomial regression to the
algorithm described in Section~\ref{sec:algo}. The only difference
between the two methods lies in the way of approximating the
conditional expectation at each time step. The two algorithms are
implemented in Python using the \emph{PolynomialFeatures} toolbox of
\emph{scikit-learn} (\cite{scikit-learn}) for the polynomial regression
and the \emph{tensorFlow} toolbox (\cite{tf2015}) to compute the neural
network approximation. We have chosen options for which there is a
substantial gap between the European and Bermudan prices, which means
that there exists indeed an early exercise strategy and that the
accuracy of the conditional expectations approximations plays a major
role.

\paragraph{Details on the algorithm used in the experiments}
In all the experiments, we have run our algorithm $100$ times to compute the
average price along with the half-width of the confidence interval for the price
estimator reported in the tables between parentheses in the form $(\pm \cdot)$.
Although the confidence interval is informative to know how much we can trust a
price, it completely squeezes the bias related to the approximation of the
conditional expectations. Remember that the estimator given
by~\eqref{eq:price-mc} is not an unbiased estimator and one should therefore be
very careful when comparing the results. Keep in mind that a higher price does not always mean a better price. 

For the activation function $\sigma_{a}$ in~\eqref{eq:nn}, we have used
the leaky ReLU function defined by
\begin{equation*}
  \sigma_{a}(x) = \begin{cases}
    x & \text{if } x \ge 0 \\
    0.3 \, \times \, x & \text{if } x < 0
  \end{cases}
\end{equation*}
We relied on the ADAM algorithm to fit the neural network at each time
step and the columns \emph{epochs} refer to the number of times we go
through the entire data set to train the network. Note that using
\emph{epochs} $= 1$ corresponds to the standard approach used in
online stochastic approximation, in which each data is used only
once. We use the same neural network through all the time steps and in
particular at a time step $0< n < N-1$, we take the optimal parameter
at time $n+1$, $\htheta_{n+1}^{p,M}$, as the starting point of the
training algorithm. Because of this smart choice, there is actually no
use setting \emph{epochs}$ > 1$ for $n < N-1$. We observed in our
numerical experiments that passing over all the data several times
does not reduce the training error at times $0 < n < N-1$, whereas it
does help when fitting the first neural network at time $N-1$. This
allows for huge computational time savings.

For learning the continuation value at each exercising date, we only use the
in-the-money paths as already suggested in the original Longstaff Schwartz
algorithm~\cite{LS01}. This means that the definition of the optimization
problem~\eqref{eq:theta_p} has to be changed into
\begin{align*}
  \inf_{\theta \in \Theta_p} \E\left[ \abs{\Phi_p(X_{T_n}; \theta) -
  Z_{\tau_{n+1}^p}}^2 \ind{Z_{T_n} > 0} \right].
\end{align*}
The empirical counterpart~\eqref{eq:theta_pM} needs to be adapted in a similar way.
Note that it does not change the theoretical analysis of the algorithm but it
is numerically more efficient. We proceed similarly in the original Longstaff
Schwartz algorithm we are comparing to in the next sections.

\subsection{Examples in the Black Scholes model}

The $d-$dimensional Black Scholes model writes for $j \in \{1, \dots, d\}$
\begin{align*}
  dS^j_t = S^j_t ( r_t dt + \sigma^j L_j dB_t)
\end{align*}
where $B$ is a Brownian motion with values in $\R^d$,
$\sigma = (\sigma^1, \dots, \sigma^d)$ is the vector of volatilities,
assumed to be deterministic and positive at all times and $L_j$ is the
$j$-th row of the matrix $L$ defined as a square root of the
correlation matrix $\Gamma$, given by
\begin{equation*}
  \Gamma = \begin{pmatrix}
    1 & \rho & \hdots & \rho\\
    \rho & 1 &\ddots & \vdots\\
    \vdots&\ddots&\ddots& \rho\\
    \rho &\hdots & \rho & 1
  \end{pmatrix}
\end{equation*}
where $\rho \in ]-1 / (d-1), 1]$ to ensure that $\Gamma$ is positive definite.

\subsubsection{Benchmarking the method on the one-dimensional put option}

Before investigating more elaborate numerical examples, we want to test our
method on the one dimensional put option. As standard as this example might be,
getting a trustworthy reference price is not an easy task. For this example, we
compare our approach to the benchmark price computed by a convolution method
in~\cite{Oosterlee08} and later used as a reference price
in~\cite{oosterlee09}. Their reference price is $11.987$ where all the digits
are accurate.

\begin{table}[htp]
  \centering\begin{tabular}{cc|ccc}
    \hline
    $L$ & $d_l$ & epochs=1 & epochs=5 & epochs=10 \\
    \hline
    2 & 32 & 11.96 ($\pm$ 0.07) & 11.97 ($\pm$ 0.06) & 11.98 ($\pm$ 0.057) \\
    2 & 128 & 11.96 ($\pm$ 0.07) & 11.97 ($\pm$ 0.056) & 11.97 ($\pm$ 0.061) \\
    2 & 512 & 11.95 ($\pm$ 0.076) & 11.95 ($\pm$ 0.08) & 11.96 ($\pm$ 0.071) \\
    4 & 32 & 11.93 ($\pm$ 0.083) & 11.94 ($\pm$ 0.09) & 11.96 ($\pm$ 0.075) \\
    4 & 128 & 11.89 ($\pm$ 0.145) & 11.93 ($\pm$ 0.097) & 11.95 ($\pm$ 0.081) \\
    4 & 512 & 11.86 ($\pm$ 0.127) & 11.93 ($\pm$ 0.096) & 11.94 ($\pm$ 0.072) \\
    8 & 32 & 11.89 ($\pm$ 0.12) & 11.93 ($\pm$ 0.117) & 11.95 ($\pm$ 0.096) \\
    8 & 128 & 11.88 ($\pm$ 0.126) & 11.92 ($\pm$ 0.11) & 11.94 ($\pm$ 0.102) \\
    8 & 512 & 11.85 ($\pm$ 0.129) & 11.9 ($\pm$ 0.163) & 11.92 ($\pm$ 0.111) \\
    \hline
  \end{tabular}
  \caption{Put option with $r=0.1$, $T=1$, $K=110$, $S_0=100$, $\sigma=0.25$, $N=10$ and $M=100,000$. The true price is $11.987$.}
  \label{tab:put-bs}
\end{table}

We can see from Table~\ref{tab:put-bs} that using a really small neural
network with only one input layer with $32$ intermediate neurons and
one output layer --- meaning that the activation function is applied
only once --- already yields very good results with a relative
accuracy greater than $1\%$. Increasing the number of epochs helps 
correct the bias created by the truncated approximation of the
conditional expectations. The larger the neural network (see in particular
the cases $L=8$), the more epochs we need to ensure that the fitting
procedure has sufficiently well converged in order to make the most of
the capabilities of the network to accurately approximate the
conditional expectations. Note that increasing the size of the network
also increases the overall variance of the algorithm as in the case of
a polynomial regression when the size of the regression basis
increases (see~\cite{glasserman04number} for details).

\subsubsection{A geometric basket option in the Black Scholes model}
\label{sec:numerics-geom}

Benchmarking a new method on high dimensional products becomes hardly feasible
as almost no high dimensional Bermudan options can be priced accurately in a
reasonable time. An exception to this is the geometric put option with payoff
$(K - (\prod_{j=1}^d S^j_t)^{1/d})_+$. Easy calculations show that the price of
this $d-$dimensional option equals the one of the $1-$dimensional option with
the following parameters
\begin{align*}
  \hat S_0 = \left( \prod_{j=1}^d S_0^j \right)^{1/d}; \quad
  \hat \sigma = \inv{d} \sqrt{\sigma^t \Gamma \sigma}; \quad
  \hat \delta  = \frac{1}{d} \sum_{j=1}^d \left( \delta^j + \inv{2} (\sigma^j)^2 \right)
  - \inv{2} (\hat \sigma)^2.
\end{align*}
In every numerical experiments on the geometric basket
option, we report the price of the equivalent one dimensional Bermudan
option obtained by the CRR tree method~\cite{crr} with $100,000$
discretization time steps.

\begin{table}[htbp]
  \centering
  \begin{tabular}{cc|ccc}
    \hline
    $L$ & $d_l$ & epochs=1 & epochs=5 & epochs=10 \\
    \hline
    2 & 32 & 4.55 ($\pm$ 0.038) & 4.56 ($\pm$ 0.041) & 4.56 ($\pm$ 0.031) \\
    2 & 128 & 4.55 ($\pm$ 0.032) & 4.56 ($\pm$ 0.04) & 4.56 ($\pm$ 0.038) \\
    2 & 512 & 4.54 ($\pm$ 0.04) & 4.55 ($\pm$ 0.033) & 4.55 ($\pm$ 0.041) \\
    4 & 32 & 4.52 ($\pm$ 0.044) & 4.54 ($\pm$ 0.04) & 4.55 ($\pm$ 0.036) \\
    4 & 128 & 4.52 ($\pm$ 0.044) & 4.54 ($\pm$ 0.033) & 4.55 ($\pm$ 0.041) \\
    4 & 512 & 4.5 ($\pm$ 0.046) & 4.54 ($\pm$ 0.042) & 4.54 ($\pm$ 0.045) \\
    8 & 32 & 4.52 ($\pm$ 0.043) & 4.54 ($\pm$ 0.049) & 4.55 ($\pm$ 0.052) \\
    8 & 128 & 4.51 ($\pm$ 0.046) & 4.53 ($\pm$ 0.045) & 4.54 ($\pm$ 0.045) \\
    8 & 512 & 4.47 ($\pm$ 0.181) & 4.51 ($\pm$ 0.051) & 4.52 ($\pm$ 0.149) \\
    \hline
  \end{tabular}
  \caption{Prices for the geometric basket put option with parameters
    $d=2$, $S_0^i = 100$, $\sigma^i = 0.2$, $\rho=0$, $\delta^j = 0.2$,
    $T=1$, $r = 0.05$, $K=100$, $N = 10$ and $M=100,000$. The true
    price is $4.57$ (all digits are significant).} \label{tab:geom-2}
\end{table}

We can see from the numerical results of Table~\ref{tab:geom-2} that even a
small neural network is able to capture the continuation values very well.
Increasing the size of the network does not help get a better price
but increases the variance unless we ensure a very accurate fit of the network
by going through the data several times (see the column \emph{epochs=10} for
instance), which in turn leads to a much larger computational cost. In
comparison, the prices obtained with the standard Longstaff Schwartz algorithm
with polynomial regressions of order respectively $1$, $3$ and $6$ are $4.47\,
\pm 0.015$, $4.56\, \pm 0.02$ and $4.57 \,\pm 0.017$. On this small
dimensional example, a low degree polynomial as well as a small neural network
give a very accurate price.

\begin{table}[htbp]
  \centering
  \begin{tabular}{cc|ccc}
    \hline
    $L$ & $d_l$ & epochs=1 & epochs=5 & epochs=10 \\
    \hline
    2 & 32 & 2.91 ($\pm$ 0.027) & 2.92 ($\pm$ 0.023) & 2.93 ($\pm$ 0.019) \\
    2 & 128 & 2.91 ($\pm$ 0.025) & 2.93 ($\pm$ 0.021) & 2.94 ($\pm$ 0.024) \\
    2 & 512 & 2.9 ($\pm$ 0.025) & 2.93 ($\pm$ 0.023) & 2.94 ($\pm$ 0.027) \\
    4 & 32 & 2.9 ($\pm$ 0.027) & 2.92 ($\pm$ 0.029) & 2.94 ($\pm$ 0.021) \\
    4 & 128 & 2.9 ($\pm$ 0.033) & 2.92 ($\pm$ 0.023) & 2.93 ($\pm$ 0.027) \\
    4 & 512 & 2.89 ($\pm$ 0.028) & 2.91 ($\pm$ 0.033) & 2.93 ($\pm$ 0.033) \\
    8 & 32 & 2.9 ($\pm$ 0.02) & 2.92 ($\pm$ 0.029) & 2.94 ($\pm$ 0.024) \\
    8 & 128 & 2.9 ($\pm$ 0.036) & 2.92 ($\pm$ 0.026) & 2.94 ($\pm$ 0.026) \\
    8 & 512 & 2.88 ($\pm$ 0.042) & 2.91 ($\pm$ 0.033) & 2.92 ($\pm$ 0.034) \\
    \hline
  \end{tabular}
  \caption{Prices for the geometric basket put option with parameters
    $d=10$, $S_0^i = 100$, $\sigma^i = 0.2$, $\rho=0.2$,
    $\delta^j = 0$, $T=1$, $r = 0.05$, $K=100$, $n = 10$ and
    $M=100,000$. The true price is $2.92$.} \label{tab:geom-10}
\end{table}

The numerical results for the $10-$dimensional geometric put option (see
Table~\ref{tab:geom-10}) show the same behavior as the low dimensional
problem. Using a small neural network provides very accurate results
within $0.1\%$ of the true price. Passing several times over the data to train
the network helps a little reduce the bias of the price estimator but at the
expense of a much higher computational effort. In comparison, the prices
obtained with the standard Longstaff Schwartz algorithm with polynomial
regression of order respectively $1$ and $3$ are $2.86\, \pm 0.014$ and $2.96\,
\pm 0.014$. Note that a regression of order $6$ is unreachable in dimension
$10$. Unlike all other examples, in which the standard Longstaff Schwartz
algorithm with polynomial regression tends to exhibit a systematic negative
bias, increasing the polynomial degree in this example yields a price above the
true one. Note that the true price is always within the confidence intervals
reported in Table~\ref{tab:geom-10}. Our method does not seem to suffer from
this positive bias phenomenon.

\begin{table}[htbp]
  \centering \begin{tabular}{cc|ccc}
    \hline
    $L$ & $d_l$ & epochs=1 & epochs=5 & epochs=10 \\
    \hline
    2 & 128 & 2.52 ($\pm$ 0.025) & 2.57 ($\pm$ 0.019) & 2.61 ($\pm$ 0.021) \\
    2 & 256 & 2.51 ($\pm$ 0.027) & 2.57 ($\pm$ 0.018) & 2.61 ($\pm$ 0.017) \\
    2 & 512 & 2.5 ($\pm$ 0.011) & 2.56 ($\pm$ 0.021) & 2.61 ($\pm$ 0.023) \\
    4 & 128 & 2.51 ($\pm$ 0.03) & 2.59 ($\pm$ 0.023) & 2.78 ($\pm$ 0.045) \\
    4 & 256 & 2.51 ($\pm$ 0.031) & 2.57 ($\pm$ 0.018) & 2.75 ($\pm$ 0.023) \\
    4 & 512 & 2.49 ($\pm$ 0.02) & 2.55 ($\pm$ 0.025) & 2.65 ($\pm$ 0.035) \\
    8 & 128 & 2.51 ($\pm$ 0.018) & 2.58 ($\pm$ 0.022) & 2.76 ($\pm$ 0.051) \\
    8 & 256 & 2.51 ($\pm$ 0.026) & 2.57 ($\pm$ 0.021) & 2.75 ($\pm$ 0.038) \\
    8 & 512 & 2.46 ($\pm$ 0.135) & 2.56 ($\pm$ 0.021) & 2.65 ($\pm$ 0.056) \\
    \hline
  \end{tabular}
  \caption{Prices for the geometric basket put option with parameters $d=40$, $S_0^i = 100$, $\sigma^i = 0.2$, $\rho=0.2$, $\delta^j = 0$, $T=1$, $r = 0.5$, $K=100$, $n = 10$ and $M=100,000$. The true price is $2.52$.} \label{tab:geom-50-1E5}
\end{table}

\begin{table}[htbp]
  \centering \begin{tabular}{cc|ccc}
    \hline
    $L$ & $d_l$ & epochs=1 & epochs=5 & epochs=10 \\
    \hline
    2 & 128 & 2.5 ($\pm$ 0.008) & 2.52 ($\pm$ 0.005) & 2.52 ($\pm$ 0.005) \\
    2 & 256 & 2.5 ($\pm$ 0.012) & 2.52 ($\pm$ 0.005) & 2.52 ($\pm$ 0.01) \\
    2 & 512 & 2.49 ($\pm$ 0.015) & 2.51 ($\pm$ 0.007) & 2.52 ($\pm$ 0.009) \\
    4 & 128 & 2.5 ($\pm$ 0.006) & 2.52 ($\pm$ 0.006) & 2.53 ($\pm$ 0.003) \\
    4 & 256 & 2.5 ($\pm$ 0.009) & 2.51 ($\pm$ 0.007) & 2.52 ($\pm$ 0.005) \\
    4 & 512 & 2.49 ($\pm$ 0.008) & 2.51 ($\pm$ 0.011) & 2.52 ($\pm$ 0.014) \\
    8 & 128 & 2.5 ($\pm$ 0.007) & 2.53 ($\pm$ 0.011) & 2.54 ($\pm$ 0.008) \\
    8 & 256 & 2.49 ($\pm$ 0.012) & 2.52 ($\pm$ 0.011) & 2.53 ($\pm$ 0.007) \\
    8 & 512 & 2.46 ($\pm$ 0.154) & 2.49 ($\pm$ 0.053) & 2.51 ($\pm$ 0.015) \\
    \hline
  \end{tabular}
  \caption{Prices for the geometric basket put option with parameters $d=40$, $S_0^i = 100$, $\sigma^i = 0.2$, $\rho=0.2$, $\delta^j = 0$, $T=1$, $r = 0.5$, $K=100$, $n = 10$ and $M=1,000,000$. The true price is $2.52$} \label{tab:geom-50-1E6}
\end{table}

Finally, we tested our approach on a $40-$dimensional geometric basket option with two different number of Monte Carlo samples $M=100,000$ (Table~\ref{tab:geom-50-1E5}) and $M=1,000,000$ (Table~\ref{tab:geom-50-1E6}). In both cases, the results obtained with really small neural networks ($L=2$) are already very accurate. However, one should note that increasing the number of epochs with $M=100,000$ leads to upper biased prices. This is the result of an overfitting phenomenon during the neural network calibration. Remember that the number of parameters of the network is $(1+d)(1+d_l) + (d_l)^{2(L - 2)}$. For $L=4$ and $d_l=128$, it already gives more than 26 million parameters. The learning capabilities of the network are such that it also learns the noise inside the data. Increasing the size of the data set ($M=1,000,000$) fixes this issue as one can see in Table~\ref{tab:geom-50-1E6}.

\subsubsection{A put basket option}

We consider a put basket option with payoff
\[
  \left(K - \sum_{i=1}^d \omega_i S_T^i\right)_+.
\]

\begin{table}[htp]
  \centering\begin{tabular}{cc|ccc}
    \hline
    $L$ & $d_l$ & epochs=1 & epochs=5 & epochs=10 \\
    \hline
    2 & 32 & 3.08 ($\pm$ 0.023) & 3.09 ($\pm$ 0.023) & 3.1 ($\pm$ 0.028) \\
    2 & 128 & 3.08 ($\pm$ 0.024) & 3.09 ($\pm$ 0.024) & 3.1 ($\pm$ 0.027) \\
    2 & 512 & 3.08 ($\pm$ 0.032) & 3.09 ($\pm$ 0.023) & 3.09 ($\pm$ 0.03) \\
    4 & 32 & 3.07 ($\pm$ 0.032) & 3.09 ($\pm$ 0.031) & 3.1 ($\pm$ 0.027) \\
    4 & 128 & 3.07 ($\pm$ 0.03) & 3.09 ($\pm$ 0.027) & 3.09 ($\pm$ 0.027) \\
    4 & 512 & 3.06 ($\pm$ 0.038) & 3.08 ($\pm$ 0.031) & 3.09 ($\pm$ 0.03) \\
    8 & 32 & 3.07 ($\pm$ 0.032) & 3.09 ($\pm$ 0.028) & 3.09 ($\pm$ 0.033) \\
    8 & 128 & 3.06 ($\pm$ 0.035) & 3.08 ($\pm$ 0.026) & 3.1 ($\pm$ 0.027) \\
    8 & 512 & 3.06 ($\pm$ 0.053) & 3.07 ($\pm$ 0.053) & 3.08 ($\pm$ 0.038) \\
    \hline
  \end{tabular}
  \caption{Basket option with $d=5$, $r=0.05$, $T=1$, $\sigma^i = 0.2$, $\omega^i = 1/d$, $S_0^i=100$, $\rho=0.2$, $K=100$, $N=10$ and $M = 100,000$.} \label{tab:basket-d5}
\end{table}

We test our algorithm in dimension $5$ and
report the results in Table~\ref{tab:basket-d5}. The standard Longstaff
Schwartz algorithm yields $3.11 \, \pm \,0.01 $ (resp.
$3.05 \, \pm \, 0.01$) for an order $3$ (resp. $1$) polynomial
regression. The prices reported in
Table~\ref{tab:basket-d5} are very close to the one obtained with an
order $3$ polynomial regression. We can see that using a very large
neural networks with several hidden layers and several hundreds of
neurons per layer does not really help. The results obtained for a
small network with a few dozens of neurons are already very good. The
difference between the results for $1$ epoch and $10$ epochs is about
half the width of the confidence interval, which makes it non
meaningful. Hence, there is no use putting more computational effort to
go through all the data set more than once.

Now, we turn to a high-dimensional problem and consider a call option on a
basket with $40$ assets to test the scalability of a our approach. The results
are reported in Table~\ref{tab:basket-d40-M1E5} for $M=100,000$ Monte Carlo
samples and in Table~\ref{tab:basket-d40-M1E6} for $M=1,000,000$ Monte Carlo
samples. As a comparison, \cite{GZvar2019} reported prices between $2.15$ and
$2.22$ for the same option. Our prices lie in this interval. We note that
increasing the number of epochs for a relatively small number of Monte Carlo
samples $M=100,000$ gives larger prices. This is all the more striking as the
size of the neural network is large, which clearly exhibits an over-fitting
phenomenon. Indeed, increasing the number of samples to $M=1,000,000$ fixes the
issue as the prices in Table~\ref{tab:basket-d40-M1E6} are between $2.14$ and
$2.18$ for all the neural network configurations and the number of epochs up to
$10$. This clearly shows that to avoid an upper bias, we need to increase the
number of samples when the size of the problem increases, which was already
noted with the standard Longstaff Scwhartz algorithm using polynomial
regression. With this example, we come to the same conclusions as for the
$40-$dimensional geometric basket option studied in
Section~\ref{sec:numerics-geom}.
\begin{table}[htp]
  \centering\begin{tabular}{cc|ccc}
    \hline
    $L$ & $d_l$ & epochs=1 & epochs=5 & epochs=10 \\
    \hline
    2 & 32 & 2.15 ($\pm$ 0.018) & 2.19 ($\pm$ 0.019) & 2.21 ($\pm$ 0.02) \\
    2 & 128 & 2.16 ($\pm$ 0.016) & 2.21 ($\pm$ 0.015) & 2.25 ($\pm$ 0.021) \\
    2 & 512 & 2.15 ($\pm$ 0.017) & 2.21 ($\pm$ 0.014) & 2.26 ($\pm$ 0.017) \\
    4 & 32 & 2.16 ($\pm$ 0.018) & 2.21 ($\pm$ 0.015) & 2.26 ($\pm$ 0.017) \\
    4 & 128 & 2.16 ($\pm$ 0.021) & 2.24 ($\pm$ 0.024) & 2.43 ($\pm$ 0.026) \\
    4 & 512 & 2.15 ($\pm$ 0.018) & 2.2 ($\pm$ 0.025) & 2.31 ($\pm$ 0.026) \\
    8 & 32 & 2.17 ($\pm$ 0.028) & 2.21 ($\pm$ 0.02) & 2.28 ($\pm$ 0.023) \\
    8 & 128 & 2.16 ($\pm$ 0.026) & 2.24 ($\pm$ 0.025) & 2.41 ($\pm$ 0.032) \\
    8 & 512 & 2.14 ($\pm$ 0.064) & 2.19 ($\pm$ 0.031) & 2.29 ($\pm$ 0.044) \\
    \hline
  \end{tabular}
  \caption{Basket option with $d=40$, $r=0.05$, $T=1$, $\sigma^i = 0.2$, $\omega^i = 1/d$, $S_0^i=100$, $\rho=0.2$, $K=100$, $N=10$ and $M = 100,000$.} \label{tab:basket-d40-M1E5}
\end{table}

\begin{table}[htp]
  \centering\begin{tabular}{cc|ccc}
    \hline
    $L$ & $d_l$ & epochs=1 & epochs=5 & epochs=10 \\
    \hline
    2 & 32 & 2.16 ($\pm$ 0.008) & 2.17 ($\pm$ 0.008) & 2.18 ($\pm$ 0.009) \\
    2 & 128 & 2.16 ($\pm$ 0.009) & 2.17 ($\pm$ 0.008) & 2.17 ($\pm$ 0.007) \\
    2 & 512 & 2.15 ($\pm$ 0.01) & 2.17 ($\pm$ 0.007) & 2.17 ($\pm$ 0.005) \\
    4 & 32 & 2.17 ($\pm$ 0.008) & 2.17 ($\pm$ 0.008) & 2.18 ($\pm$ 0.007) \\
    4 & 128 & 2.16 ($\pm$ 0.012) & 2.17 ($\pm$ 0.008) & 2.18 ($\pm$ 0.007) \\
    4 & 512 & 2.15 ($\pm$ 0.014) & 2.16 ($\pm$ 0.01) & 2.16 ($\pm$ 0.01) \\
    8 & 32 & 2.16 ($\pm$ 0.011) & 2.18 ($\pm$ 0.009) & 2.18 ($\pm$ 0.006) \\
    8 & 128 & 2.16 ($\pm$ 0.015) & 2.17 ($\pm$ 0.007) & 2.18 ($\pm$ 0.007) \\
    8 & 512 & 2.14 ($\pm$ 0.022) & 2.15 ($\pm$ 0.056) & 2.16 ($\pm$ 0.015) \\
    \hline
  \end{tabular}
  \caption{Basket option with $d=40$, $r=0.05$, $T=1$, $\sigma^i = 0.2$, $\omega^i = 1/d$, $S_0^i=100$, $\rho=0.2$, $K=100$, $N=10$ and $M = 1,000,000$.} \label{tab:basket-d40-M1E6}
\end{table}

\subsubsection{A call on the maximum of several assets in the Black Scholes model}

We consider a call option on the maximum of $d$ assets in the Black Scholes model with payoff
\begin{equation*}
  \left( \max_{i=1,\dots,d} S_T^i - K \right)_+.
\end{equation*}
The different sets of parameters are chosen as in~\cite{becker2019deep,GZvar2019} to easily compare the prices obtained with the different methods.

\begin{table}[htp]
  \centering\begin{tabular}{cc|cccc}
    \hline
    $L$ & $d_l$ & epochs=1 & epochs=5 & epochs=10 \\
    \hline
    2 & 32 & 25.97 ($\pm$ 0.117) & 25.95 ($\pm$ 0.141) & 25.94 ($\pm$ 0.133) \\
    2 & 128 & 25.95 ($\pm$ 0.11) & 25.95 ($\pm$ 0.126) & 26.02 ($\pm$ 0.113) \\
    2 & 512 & 25.92 ($\pm$ 0.104) & 25.96 ($\pm$ 0.116) & 26.01 ($\pm$ 0.153) \\
    4 & 32 & 25.83 ($\pm$ 0.132) & 25.97 ($\pm$ 0.146) & 26.02 ($\pm$ 0.139) \\
    4 & 128 & 25.76 ($\pm$ 0.203) & 25.91 ($\pm$ 0.162) & 25.99 ($\pm$ 0.162) \\
    4 & 512 & 25.63 ($\pm$ 0.238) & 25.85 ($\pm$ 0.181) & 25.94 ($\pm$ 0.146) \\
    8 & 32 & 25.72 ($\pm$ 0.185) & 25.91 ($\pm$ 0.134) & 25.96 ($\pm$ 0.169) \\
    8 & 128 & 25.61 ($\pm$ 0.251) & 25.84 ($\pm$ 0.186) & 25.93 ($\pm$ 0.143) \\
    8 & 512 & 25.49 ($\pm$ 0.265) & 25.76 ($\pm$ 0.223) & 25.83 ($\pm$ 0.2) \\
    \hline
  \end{tabular}
  \caption{Prices for the call option on the maximum of $5$ assets with
  parameters $S_0^j = 100$, $T=3$, $r = 0.05$, $K=100$, $\rho=0$, $\sigma^j = 0.2$, $\delta^j = 0.1$, $N = 9$ and $M=100,000$. } \label{tab:max-bs-d5}
\end{table}
Pricing a call on the maximum of a basket of assets is usually far more
difficult than a standard basket option because of the strong non linearity of
the maximum function. For the example of Table~\ref{tab:max-bs-d5}, the standard
Longstaff Schwartz algorithm yields $25.34 \, \pm \,0.06 $, $25.98 \, \pm \,
0.05$, $26.29 \, \pm \, 0.04$ for a polynomial regression of order $1$, $3$ and
$6$ respectively. The prices obtained with the polynomial regression vary a lot
with the degree of the regression. For this example, \cite{becker2019deep}
reported a $95\%$ confidence interval of $[26.14, 26.17]$. The prices reported
in Table~\ref{tab:max-bs-d5} are very close to this confidence interval.
A small neural network ($L=2$) enables us to get values within $1\%$ of
the true price, which is a great achievement considering the complexity of the
product and the small size of the approximation. As in the other examples, using
several passes through the data to train the neural network does not really
bring any improvement for small neural networks. For larger networks, it helps a
little but in the end larger networks are less accurate than smaller ones. To
get the best of larger neural networks, we would need more data to train the
networks, ie. more Monte Carlo samples as we already observed for the high dimensional geometric put option.

\begin{table}[htp]
  \centering\begin{tabular}{cc|cccc}
    \hline
    $L$ & $d_l$ & epochs=1 & epochs=5 & epochs=10 \\
    \hline
    2 & 128 & 68.99 ($\pm$ 0.179) & 69.26 ($\pm$ 0.164) & 69.42 ($\pm$ 0.169) \\
    2 & 256 & 69.07 ($\pm$ 0.149) & 69.42 ($\pm$ 0.125) & 69.45 ($\pm$ 0.138) \\
    2 & 512 & 69.11 ($\pm$ 0.194) & 69.43 ($\pm$ 0.18) & 69.51 ($\pm$ 0.167) \\
    4 & 128 & 68.91 ($\pm$ 0.365) & 69.29 ($\pm$ 0.334) & 69.55 ($\pm$ 0.339) \\
    4 & 256 & 68.72 ($\pm$ 0.358) & 69.24 ($\pm$ 0.341) & 69.5 ($\pm$ 0.369) \\
    4 & 512 & 68.54 ($\pm$ 0.548) & 69.17 ($\pm$ 0.356) & 69.34 ($\pm$ 0.359) \\
    8 & 128 & 68.59 ($\pm$ 0.613) & 69.32 ($\pm$ 0.348) & 69.71 ($\pm$ 0.497) \\
    8 & 256 & 68.57 ($\pm$ 0.797) & 69.25 ($\pm$ 0.564) & 69.4 ($\pm$ 0.484) \\
    8 & 512 & 68.32 ($\pm$ 1.444) & 69.01 ($\pm$ 0.738) & 69.49 ($\pm$ 0.487) \\
    \hline
  \end{tabular}
  \caption{Prices for the call option on the maximum of $50$ assets with
  parameters $S_0^j = 100$, $T=3$, $r = 0.05$, $K=100$, $\rho=0$, $\sigma^j = 0.2$, $\delta^j = 0.1$, $N = 9$ and $M=100,000$. } \label{tab:max-bs-d50-1E5}
\end{table}

\begin{table}[htp]
  \centering\begin{tabular}{cc|cccc}
    \hline
    $L$ & $d_l$ & epochs=1 & epochs=5 & epochs=10 \\
    \hline
    2 & 128 & 68.85 ($\pm$ 0.074) & 68.96 ($\pm$ 0.095) & 69.01 ($\pm$ 0.119) \\
    2 & 256 & 68.87 ($\pm$ 0.1) & 69.0 ($\pm$ 0.143) & 69.07 ($\pm$ 0.146) \\
    2 & 512 & 68.82 ($\pm$ 0.082) & 69.05 ($\pm$ 0.128) & 69.19 ($\pm$ 0.136) \\
    4 & 128 & 68.84 ($\pm$ 0.221) & 69.28 ($\pm$ 0.153) & 69.41 ($\pm$ 0.211) \\
    4 & 256 & 68.75 ($\pm$ 0.342) & 69.14 ($\pm$ 0.296) & 69.38 ($\pm$ 0.342) \\
    4 & 512 & 68.7 ($\pm$ 0.426) & 69.05 ($\pm$ 0.317) & 69.35 ($\pm$ 0.254) \\
    8 & 128 & 68.81 ($\pm$ 0.277) & 69.28 ($\pm$ 0.291) & 69.64 ($\pm$ 0.22) \\
    8 & 256 & 68.57 ($\pm$ 0.512) & 69.34 ($\pm$ 0.378) & 69.65 ($\pm$ 0.414) \\
    \hline
  \end{tabular}
  \caption{Prices for the call option on the maximum of $50$ assets with
  parameters $S_0^j = 100$, $T=3$, $r = 0.05$, $K=100$, $\rho=0$, $\sigma^j = 0.2$, $\delta^j = 0.1$, $N = 9$ and $M=1,000,000$. } \label{tab:max-bs-d50-1E6}
\end{table}

We tested the scalability of our approach on a $50-$dimensional max-call option.
The results are reported in Table~\ref{tab:max-bs-d50-1E5} for $M=100,000$ and
Table~\ref{tab:max-bs-d50-1E6} for $M=1,000,000$. \cite{becker2019deep} report
$[69.56, 69.95]$ as the $95\%$ confidence interval for the option price. We
obtain prices very close these values. As a comparison, the standard Longstaff
Schwartz algorithm yields $67.96 \, \pm \,0.02 $ and $69.02 \, \pm \, 0.02$ for
a polynomial regression of order $1$ and $2$ respectively. The order $3$
regression is out of reach. Our deep learning approach scales far better than
the standard polynomial regression. Increasing the number of epochs improves the result, which means that the neural network has to be much more finely tuned than in the other examples studied sofar.

\subsection{A put option in the Heston model}

We consider the Heston model defined by
\begin{align*}
  dS_t &= S_t(r_t dt + \sqrt{\sigma_t} (\rho dW^1_t + \sqrt{1 - \rho^2} dW^2_t)) \\
  d\sigma_t &=  \kappa (\theta - \sigma_t) dt + \xi \sqrt{\sigma_t} dW^1_t.
\end{align*}
For the simulation of the model, we use a modified Euler scheme with $30$ time steps per year, in which we have replaced $\sqrt{\sigma_t}$ by $\sqrt{(\sigma_t)_+}$ to deal with possibly negative values of the discretized volatility process. For the option of Table~\ref{tab:put-heston}, the standard Longstaff Schwartz algorithm yields $1.70 \, \pm \, 0.008$ (resp. $1.675 \, \pm \, 0.005$) for an order $6$ (resp. $1$) polynomial regression. As in the other examples, the use of a neural network as the regressor provides very accurate results even with a quite small network (no hidden layer and very few neurons, see the case $L=2$ and $d_l=32$).

\begin{table}[htbp]
  \centering
  \begin{tabular}{cc|ccc}
    \hline
    $L$ & $d_l$ & epochs=1 & epochs=5 & epochs=10 \\
    \hline
    2 & 32 & 1.69 ($\pm$ 0.017) & 1.7 ($\pm$ 0.017) & 1.7 ($\pm$ 0.016) \\
    2 & 128 & 1.69 ($\pm$ 0.017) & 1.7 ($\pm$ 0.019) & 1.7 ($\pm$ 0.019) \\
    2 & 512 & 1.69 ($\pm$ 0.019) & 1.69 ($\pm$ 0.019) & 1.69 ($\pm$ 0.018) \\
    4 & 32 & 1.69 ($\pm$ 0.022) & 1.69 ($\pm$ 0.017) & 1.7 ($\pm$ 0.018) \\
    4 & 128 & 1.69 ($\pm$ 0.024) & 1.69 ($\pm$ 0.02) & 1.7 ($\pm$ 0.016) \\
    4 & 512 & 1.68 ($\pm$ 0.025) & 1.69 ($\pm$ 0.022) & 1.69 ($\pm$ 0.022) \\
    8 & 32 & 1.69 ($\pm$ 0.023) & 1.69 ($\pm$ 0.02) & 1.69 ($\pm$ 0.019) \\
    8 & 128 & 1.68 ($\pm$ 0.03) & 1.69 ($\pm$ 0.022) & 1.69 ($\pm$ 0.02) \\
    8 & 512 & 1.68 ($\pm$ 0.03) & 1.68 ($\pm$ 0.041) & 1.68 ($\pm$ 0.053) \\
    \hline
  \end{tabular}
  \caption{Prices for put option in the Heston model with parameters the geometric basket put option with parameters with $S_0 = K = 100$, $T=1$, $\sigma_0 = 0.01$, $\xi=0.2$, $\theta=0.01$, $\kappa = 2$, $\rho = -0.3$, $r=0.1$, $N=10$ and $M=100,000$. } \label{tab:put-heston}
\end{table}

\section{Conclusion}

The difficulties in pricing Bermudan options come from approximating the
continuation value at each exercising date. While polynomial regression is
widely used for this step, we have investigated the use of deep learning. We
have proved the theoretical convergence of our algorithm with respect to both
the neural network and Monte Carlo approximations. Our numerical experiments
show that the prices computed using our approach are very similar to those
obtained from the standard Longstaff Schwartz algorithm. With no surprise, using
neural networks does not help much for low dimensional problems but does scale
far better on high dimensional problems as it does not suffer from the curse of
dimensionality as much as polynomial regression does. Polynomial regression
requires a relatively high order to provide accurate prices, which is not
feasible in high dimensional problems. Neural networks approximation
capabilities seem far better and relatively small networks already provided very
accurate results. Indeed, a few hundred neurons with no hidden layers were
sufficient to have very accurate prices. Training a neural network usually
requires several passes through the whole data set. Yet, in our examples this
seemed pretty much useless mostly because the functional representation of the
continuation function should not vary much over time. So, once the neural
network has been well trained, one pass over the data ($epochs =1$) is enough to
fit the network at a new date. This saves a lot of computational time. Neural
networks have proved to be a very versatile and efficient tool to compute
Bermudan option prices especially when the problem is highly non linear.

\bibliographystyle{abbrvnat}
\bibliography{biblio.bib}

\begin{thebibliography}{37}
\providecommand{\natexlab}[1]{#1}
\providecommand{\url}[1]{\texttt{#1}}
\expandafter\ifx\csname urlstyle\endcsname\relax
  \providecommand{\doi}[1]{doi: #1}\else
  \providecommand{\doi}{doi: \begingroup \urlstyle{rm}\Url}\fi

\bibitem[Abadi et~al.(2015)Abadi, Agarwal, Barham, Brevdo, Chen, Citro,
  Corrado, Davis, Dean, Devin, Ghemawat, Goodfellow, Harp, Irving, Isard, Jia,
  Jozefowicz, Kaiser, Kudlur, Levenberg, Man\'{e}, Monga, Moore, Murray, Olah,
  Schuster, Shlens, Steiner, Sutskever, Talwar, Tucker, Vanhoucke, Vasudevan,
  Vi\'{e}gas, Vinyals, Warden, Wattenberg, Wicke, Yu, and Zheng]{tf2015}
M.~Abadi, A.~Agarwal, P.~Barham, E.~Brevdo, Z.~Chen, C.~Citro, G.~S. Corrado,
  A.~Davis, J.~Dean, M.~Devin, S.~Ghemawat, I.~Goodfellow, A.~Harp, G.~Irving,
  M.~Isard, Y.~Jia, R.~Jozefowicz, L.~Kaiser, M.~Kudlur, J.~Levenberg,
  D.~Man\'{e}, R.~Monga, S.~Moore, D.~Murray, C.~Olah, M.~Schuster, J.~Shlens,
  B.~Steiner, I.~Sutskever, K.~Talwar, P.~Tucker, V.~Vanhoucke, V.~Vasudevan,
  F.~Vi\'{e}gas, O.~Vinyals, P.~Warden, M.~Wattenberg, M.~Wicke, Y.~Yu, and
  X.~Zheng.
\newblock {TensorFlow}: Large-scale machine learning on heterogeneous systems,
  2015.
\newblock URL \url{http://tensorflow.org/}.
\newblock Software available from tensorflow.org.

\bibitem[Albertini and Sontag(1994)]{albertini94}
F.~Albertini and E.~D. Sontag.
\newblock Uniqueness of weights for recurrent nets.
\newblock \emph{MATHEMATICAL RESEARCH}, 79:\penalty0 599--599, 1994.

\bibitem[Albertini et~al.(1993)Albertini, Sontag, and Maillot]{albertini93}
F.~Albertini, E.~D. Sontag, and V.~Maillot.
\newblock Uniqueness of weights for neural networks.
\newblock \emph{Artificial Neural Networks for Speech and Vision}, pages
  115--125, 1993.

\bibitem[Arnold(2009)]{arnold57}
V.~I. Arnold.
\newblock On functions of three variables.
\newblock \emph{Collected Works: Representations of Functions, Celestial
  Mechanics and KAM Theory, 1957--1965}, pages 5--8, 2009.

\bibitem[Bally and Pages(2003)]{bapa03}
V.~Bally and G.~Pages.
\newblock A quantization algorithm for solving multidimensional discrete-time
  optimal stopping problems.
\newblock \emph{Bernoulli}, 9\penalty0 (6):\penalty0 1003--1049, 2003.

\bibitem[Becker et~al.(2019{\natexlab{a}})Becker, Cheridito, and
  Jentzen]{becker2019deep}
S.~Becker, P.~Cheridito, and A.~Jentzen.
\newblock Deep optimal stopping.
\newblock \emph{Journal of Machine Learning Research}, 20\penalty0
  (74):\penalty0 1--25, 2019{\natexlab{a}}.

\bibitem[Becker et~al.(2019{\natexlab{b}})Becker, Cheridito, Jentzen, and
  Welti]{becker2019deep2}
S.~Becker, P.~Cheridito, A.~Jentzen, and T.~Welti.
\newblock Solving high-dimensional optimal stopping problems using deep
  learning, 2019{\natexlab{b}}.

\bibitem[Bottou et~al.(2018)Bottou, Curtis, and
  Nocedal]{bottou2018optimization}
L.~Bottou, F.~E. Curtis, and J.~Nocedal.
\newblock Optimization methods for large-scale machine learning.
\newblock \emph{Siam Review}, 60\penalty0 (2):\penalty0 223--311, 2018.

\bibitem[Bouchard and Warin(2012)]{BoucharWarin12}
B.~Bouchard and X.~Warin.
\newblock Monte-carlo valuation of american options: Facts and new algorithms
  to improve existing methods.
\newblock In R.~A. Carmona, P.~Del~Moral, P.~Hu, and N.~Oudjane, editors,
  \emph{Numerical Methods in Finance}, volume~12 of \emph{Springer Proceedings
  in Mathematics}, pages 215--255. Springer Berlin Heidelberg, 2012.

\bibitem[Broadie and Glasserman(2004)]{brgl04}
M.~Broadie and P.~Glasserman.
\newblock A stochastic mesh method for pricing high-dimensional american
  options.
\newblock \emph{Journal of Computational Finance}, 7:\penalty0 35--72, 2004.

\bibitem[Bronstein et~al.(2013)Bronstein, Pag{\`e}s, and
  Port{\`e}s]{pages13_quantization}
A.~L. Bronstein, G.~Pag{\`e}s, and J.~Port{\`e}s.
\newblock Multi-asset american options and parallel quantization.
\newblock \emph{Methodology and Computing in Applied Probability}, 15\penalty0
  (3):\penalty0 547--561, 2013.

\bibitem[Carriere(1996)]{carriere96}
J.~F. Carriere.
\newblock Valuation of the early-exercise price for options using simulations
  and nonparametric regression.
\newblock \emph{Insurance: mathematics and Economics}, 19\penalty0
  (1):\penalty0 19--30, 1996.

\bibitem[Cl{\'e}ment et~al.(2002)Cl{\'e}ment, Lamberton, and Protter]{clp02}
E.~Cl{\'e}ment, D.~Lamberton, and P.~Protter.
\newblock An analysis of a least squares regression method for american option
  pricing.
\newblock \emph{Finance and Stochastics}, 6\penalty0 (4):\penalty0 449--471,
  2002.

\bibitem[Cox et~al.(1979)Cox, Ross, and Rubinstein]{crr}
J.~Cox, S.~Ross, and M.~Rubinstein.
\newblock Option pricing: A simplified approach.
\newblock \emph{Journal of Financial Economics}, \penalty0 (7):\penalty0
  229--263, 1979.

\bibitem[Cybenko(1989)]{cybenko89}
G.~Cybenko.
\newblock Approximation by superpositions of a sigmoidal function.
\newblock \emph{Mathematics of Control, Signals and Systems}, 2\penalty0
  (4):\penalty0 303--314, Dec 1989.

\bibitem[E et~al.(2020)E, Ma, Wojtowytsch, and Wu]{e2020mathematical}
W.~E, C.~Ma, S.~Wojtowytsch, and L.~Wu.
\newblock Towards a mathematical understanding of neural network-based machine
  learning: what we know and what we don't, 2020.

\bibitem[Fang and Oosterlee(2009)]{oosterlee09}
F.~Fang and C.~W. Oosterlee.
\newblock Pricing early-exercise and discrete barrier options by fourier-cosine
  series expansions.
\newblock \emph{Numerische Mathematik}, 114\penalty0 (1):\penalty0 27, 2009.

\bibitem[Fehrman et~al.(2020)Fehrman, Gess, and
  Jentzen]{fehrman2020convergence}
B.~Fehrman, B.~Gess, and A.~Jentzen.
\newblock Convergence rates for the stochastic gradient descent method for
  non-convex objective functions.
\newblock \emph{Journal of Machine Learning Research}, 21\penalty0
  (136):\penalty0 1--48, 2020.

\bibitem[Ghadimi and Lan(2013)]{ghadimi2013stochastic}
S.~Ghadimi and G.~Lan.
\newblock Stochastic first-and zeroth-order methods for nonconvex stochastic
  programming.
\newblock \emph{SIAM Journal on Optimization}, 23\penalty0 (4):\penalty0
  2341--2368, 2013.

\bibitem[Glasserman and Yu(2004)]{glasserman04number}
P.~Glasserman and B.~Yu.
\newblock Number of paths versus number of basis functions in american option
  pricing.
\newblock \emph{The Annals of Applied Probability}, 14\penalty0 (4):\penalty0
  2090--2119, 2004.

\bibitem[Gobet et~al.(2005)Gobet, Lemor, and Warin]{lemor_05}
E.~Gobet, J.~Lemor, and X.~Warin.
\newblock A regression-based {M}onte {C}arlo method to solve backward
  stochastic differential equations.
\newblock \emph{Annals of Applied Probability}, 15\penalty0 (3):\penalty0
  2172--2202, 2005.

\bibitem[Gouden{\`e}ge et~al.(2019)Gouden{\`e}ge, Molent, and
  Zanette]{GZvar2019}
L.~Gouden{\`e}ge, A.~Molent, and A.~Zanette.
\newblock Variance reduction applied to machine learning for pricing
  bermudan/american options in high dimension.
\newblock \emph{arXiv preprint arXiv:1903.11275}, 2019.

\bibitem[Haugh and Kogan(2004)]{haugh-kogan}
M.~B. Haugh and L.~Kogan.
\newblock Pricing american options: a duality approach.
\newblock \emph{Operations Research}, 52\penalty0 (2):\penalty0 258--270, 2004.

\bibitem[Hornik(1991)]{hornik91}
K.~Hornik.
\newblock Approximation capabilities of multilayer feedforward networks.
\newblock \emph{Neural Networks}, 4\penalty0 (2):\penalty0 251 -- 257, 1991.

\bibitem[Kohler et~al.(2010)Kohler, Krzy{\.z}ak, and Todorovic]{kohler10}
M.~Kohler, A.~Krzy{\.z}ak, and N.~Todorovic.
\newblock Pricing of high-dimensional american options by neural networks.
\newblock \emph{Mathematical Finance: An International Journal of Mathematics,
  Statistics and Financial Economics}, 20\penalty0 (3):\penalty0 383--410,
  2010.

\bibitem[Kolmogorov(1956)]{kolmogorov56}
A.~Kolmogorov.
\newblock On the representation of continuous functions of several variables as
  superpositions of functions of smaller number of variables.
\newblock In \emph{Soviet. Math. Dokl}, volume 108, pages 179--182, 1956.

\bibitem[Ledoux and Talagrand(1991)]{ledouxtalagrand}
M.~Ledoux and M.~Talagrand.
\newblock \emph{Probability in {B}anach spaces}, volume~23 of \emph{Ergebnisse
  der Mathematik und ihrer Grenzgebiete (3) [Results in Mathematics and Related
  Areas (3)]}.
\newblock Springer-Verlag, Berlin, 1991.
\newblock ISBN 3-540-52013-9.
\newblock Isoperimetry and processes.

\bibitem[Lei et~al.(2019)Lei, Hu, Li, and Tang]{lei2019stochastic}
Y.~Lei, T.~Hu, G.~Li, and K.~Tang.
\newblock Stochastic gradient descent for nonconvex learning without bounded
  gradient assumptions.
\newblock \emph{IEEE Transactions on Neural Networks and Learning Systems},
  2019.

\bibitem[Longstaff and Schwartz(2001)]{LS01}
F.~Longstaff and R.~Schwartz.
\newblock Valuing {A}merican options by simulation : {A} simple least-square
  approach.
\newblock \emph{Review of Financial Studies}, 14:\penalty0 113--147, 2001.

\bibitem[Lord et~al.(2008)Lord, Fang, Bervoets, and Oosterlee]{Oosterlee08}
R.~Lord, F.~Fang, F.~Bervoets, and C.~W. Oosterlee.
\newblock A fast and accurate fft-based method for pricing early-exercise
  options under l{\'e}vy processes.
\newblock \emph{SIAM Journal on Scientific Computing}, 30\penalty0
  (4):\penalty0 1678--1705, 2008.

\bibitem[Pag{\`e}s(2018)]{pages18numerical}
G.~Pag{\`e}s.
\newblock \emph{Numerical Probability: An Introduction with Applications to
  Finance}.
\newblock Springer, 2018.
\newblock \doi{10.1007/978-3-319-90276-0}.

\bibitem[Pedregosa et~al.(2011)Pedregosa, Varoquaux, Gramfort, Michel, Thirion,
  Grisel, Blondel, Prettenhofer, Weiss, Dubourg, Vanderplas, Passos,
  Cournapeau, Brucher, Perrot, and Duchesnay]{scikit-learn}
F.~Pedregosa, G.~Varoquaux, A.~Gramfort, V.~Michel, B.~Thirion, O.~Grisel,
  M.~Blondel, P.~Prettenhofer, R.~Weiss, V.~Dubourg, J.~Vanderplas, A.~Passos,
  D.~Cournapeau, M.~Brucher, M.~Perrot, and E.~Duchesnay.
\newblock {Scikit-learn: Machine Learning in Python }.
\newblock \emph{Journal of Machine Learning Research}, 12:\penalty0 2825--2830,
  2011.

\bibitem[Pinkus(1999)]{pinkus1999approximation}
A.~Pinkus.
\newblock Approximation theory of the mlp model in neural networks.
\newblock \emph{Acta numerica}, 8:\penalty0 143--195, 1999.

\bibitem[Rubinstein and Shapiro(1993)]{MR1241645}
R.~Y. Rubinstein and A.~Shapiro.
\newblock \emph{Discrete event systems}.
\newblock Wiley Series in Probability and Mathematical Statistics: Probability
  and Mathematical Statistics. John Wiley \& Sons Ltd., Chichester, 1993.
\newblock ISBN 0-471-93419-4.
\newblock Sensitivity analysis and stochastic optimization by the score
  function method.

\bibitem[Tilley(1993)]{tilley93}
J.~A. Tilley.
\newblock Valuing american options in a path simulation model.
\newblock \emph{Transactions of the Society of Actuaries}, 45\penalty0
  (83):\penalty0 104, 1993.

\bibitem[Tsitsiklis and Roy(2001)]{tsit:vanr:01}
J.~Tsitsiklis and B.~V. Roy.
\newblock Regression methods for pricing complex {American}-style options.
\newblock \emph{IEEE Trans. Neural Netw.}, 12\penalty0 (4):\penalty0 694--703,
  2001.

\bibitem[Williamson and Helmke(1995)]{WH95}
R.~C. Williamson and U.~Helmke.
\newblock Existence and uniqueness results for neural network approximations.
\newblock \emph{IEEE Transactions on Neural Networks}, 6\penalty0 (1):\penalty0
  2--13, 1995.

\end{thebibliography}
\end{document}